\documentclass[12pt]{amsart}
\usepackage[margin=3 cm]{geometry}
\usepackage[stable]{footmisc}
\usepackage{graphicx}
\usepackage[shortlabels]{enumitem} 
\usepackage{cite}
\usepackage{amssymb,amsmath}
\usepackage{amsthm}
\usepackage{multirow,multicol}
\usepackage{hyperref}
\usepackage{verbatim}
\usepackage{booktabs}
\usepackage{multirow}
\usepackage{algorithm}
\usepackage[noend]{algpseudocode}

\usepackage{tikz}
\usetikzlibrary{arrows,backgrounds}
\usepackage[all]{xy}

\newcommand{\K}{{\mathbb K}}

\newcommand{\R}{{\mathbb R}}
\newcommand{\SR}{{\mathcal S}}

\newcommand{\C}{{\mathbb C}}

\newcommand{\N}{{\mathbb N}}
\newcommand{\id}{\mathrm{id}}
\newcommand{\V}{{\mathcal V}}
\newcommand{\W}{{\mathcal W}}
\newcommand{\I}{{\mathbf I}}
\newcommand*\conj[1]{\overline{#1}}

\newcommand{\trace}{\mathrm{tr}}

\DeclareMathOperator{\rank}{rank}

\renewcommand{\subset}{\subseteq}

\algnewcommand\NOT{\textbf{not\ }}
\algnewcommand\GOTO{\textbf{go\ to\ }}

\newtheorem{thm}{Theorem}
\newtheorem{prop}[thm]{Proposition}
\newtheorem{cor}[thm]{Corollary}
\newtheorem{lem}[thm]{Lemma}
\newtheorem*{prob}{Problem}

\newtheorem*{ques}{Question}

\theoremstyle{definition}
\newtheorem{defn}{Definition}[section]
\newtheorem{exmp}{Example}[section]
\theoremstyle{remark}
\newtheorem*{rem}{Remark}

\let\oldtabular\tabular
\renewcommand{\tabular}{\footnotesize\oldtabular}

\title[Sampling varieties for SOS programs]{Sampling algebraic varieties for \\Sum of Squares programs}
\date{\today}
\author{Diego Cifuentes} 
\address{
Laboratory for Information and Decision Systems (LIDS), 
Massachusetts Institute of Technology, Cambridge MA 02139, USA}
\email{diegcif@mit.edu}

\author{Pablo A. Parrilo}
\address{
Laboratory for Information and Decision Systems (LIDS), 
Massachusetts Institute of Technology, Cambridge MA 02139, USA}
\email{parrilo@mit.edu}


\begin{document}
\maketitle

\begin{abstract}
We study sum of squares (SOS) relaxations to optimize polynomial functions over a set $\V\cap \R^n$, where $\V$ is a complex algebraic variety.
We propose a new methodology that, rather than relying on some algebraic description, represents $\V$ with a generic set of complex samples.
This approach depends only on the geometry of $\V$, avoiding representation issues such as multiplicity and choice of generators. 
It also takes advantage of the coordinate ring structure to reduce the size of the corresponding semidefinite program (SDP).
In addition, the input can be given as a straight-line program.
Our methods are particularly appealing for varieties that are easy to sample from but for which the defining equations are complicated, such as $SO(n)$, Grassmannians or rank~$k$ tensors.
For arbitrary varieties we can obtain the required samples by using the tools of numerical algebraic geometry.
In this way we connect the areas of SOS optimization and numerical algebraic geometry.
\end{abstract}

\pagestyle{myheadings}
\thispagestyle{plain}
\markboth{D. Cifuentes and P.~A. Parrilo}{Sampling varieties for SOS programs}

\section{Introduction}

Consider the ring $\R[x]:=\R[x_1,\ldots,x_n]$ of multivariate polynomials and an algebraic variety~$\V\subset\C^n$.
For a given polynomial $p\in \R[x]$, we are interested in deciding whether
\begin{align}\label{eq:positive_variety}
  p(x)\geq 0 \mbox{ for all } x\in \V\cap\R^n.
\end{align}
More generally, we can consider the problem of finding lower bounds for a polynomial on a real variety.
Recall that an algebraic variety can be given implicitly, as the zero set of polynomial equations, or parametrically, as the image of $\C^n$ under a rational map.

The decision problem in~\eqref{eq:positive_variety} is computationally hard, but there are simpler relaxations based on \emph{sum of squares} (SOS) \cite{Parrilo2003}.
Recall that a polynomial $F\in \R[x]$ is SOS if it can be written in the form $F(x)=\sum_i f_i(x)^2$ for some $f_i\in \R[x]$.
Given a bound $d\in \N$, a sufficient condition for~\eqref{eq:positive_variety} to hold is the existence of a polynomial $F\in\R[x]$ such that
\begin{align}\label{eq:coordinatering}
  p(z)=F(z) \mbox{ for all } z\in\V 
  \;\;\; ( \mbox{i.e., } p \equiv F \bmod \I(\V)); \qquad
  F(x) \mbox{ is SOS}; \qquad
  \deg(F)\leq 2d.
\end{align}
We refer to such an $F$ as a \emph{$d$-SOS($\V$) certificate}.
The main problem we address in this paper is the following.

\begin{prob}
Given a bound $d\in \N$, a polynomial $p(x)$ and a variety $\V$, find a $d$-SOS($\V$) certificate (if it exists).
\end{prob}

It was shown in~\cite{Parrilo2003b} that, given a Gr\"obner basis of the ideal $\I(\V)$, the above problem reduces to a semidefinite program (SDP).
To the best of our knowledge, this is the only known method to address it.
This approach is quite effective for varieties with simple Gr\"obner bases, such as the hypercube $\{0,1\}^n$, or hypersurfaces.
Unfortunately, besides these simple cases, Gr\"obner bases computation is typically too expensive.

Given defining equations of the variety $\{h_j(x)=0\}_j$, there is a \emph{weaker} class of certificates based on writing $p(x)$ in the form $F(x) + \sum_j g_j(x)h_j(x)$ \cite{Parrilo2003}.
This approach is widely used in practice~\cite{Blekherman2013,Lasserre2009}, thanks to the convenience of allowing any set of defining equations.
But this simplicity comes with a price, since the success of the relaxation now depends on the choice of a good set of equations $\{h_j\}_j$.
Furthermore, the corresponding SDP is larger.
Indeed, for fixed~$\V$ the number of unknowns is $O(d^{\,2n})$, whereas for~\eqref{eq:coordinatering} is $O(d^{\,2\dim\V})$; see Remark~\ref{thm:hilbertfunction}.
We also point out that for several parametric varieties, notably secant varieties~\cite{Landsberg:2007aa}, the defining equations are not explicitly known, thus making this type certificates unfeasible.

\subsubsection*{Sampling certificates}
In this paper we propose an alternative geometric approach to compute SOS($\V$) certificates.
Rather than depending on an algebraic description of the variety, we rely on a generic set of samples $Z:=\{z_1,\ldots,z_S\}\subset \V$.
By specializing the condition in~\eqref{eq:coordinatering} to such samples, we get the following.

\begin{defn}\label{defn:certificate}
  Let $\V\subset \C^n$ be a variety and let $p\in \R[x]$ be nonnegative on $\V\cap\R^n$.
  Given a bound $d\in \N$, a \emph{sampling $d$-SOS pre-certificate} is a pair $(F,Z)$, where  $F(x)$ is a polynomial and $Z=\{z_1,\ldots,z_S\}\subset \V$ a sample set, such that
\begin{align}\label{eq:sdpsamplesvariety}
  p(z_s) &= F(z_s)\;\mbox{ for }s=1,\ldots,S; \qquad
  F(x) \mbox{ is SOS}; \qquad
  \deg(F)\leq 2d; 
\end{align}
  The pre-certificate is \emph{correct} if $F$ is a $d$-SOS($\V$) certificate, i.e., it satisfies ~\eqref{eq:coordinatering}.

\end{defn}

Computing a sampling pre-certificate reduces to an SDP. 
We show that suitable genericity assumptions guarantee its correctness, thus giving us an SOS($\V$) certificate.
An interesting feature of our sampling methodology is that the only information needed of the variety is a \emph{sampling oracle}, i.e., a procedure that generates generic samples.
Note that sampling points is very simple when the variety has a known parametrization (e.g., for $SO(n)$, Grassmannians, rank~$k$ tensors).
For a general variety~$\V$, the field of \emph{numerical algebraic geometry} provides practical methods to sample generic points~\cite{Sommese2005}.

\subsubsection*{Contributions}
This paper presents the following contributions.
\begin{itemize}
\item We introduce a new methodology to compute SOS certificates over an algebraic variety~$\V$.
This is a geometric formulation that represents $\V$ with a generic set of complex samples, instead of relying on some algebraic description.
In this way, we avoid algebraic issues such as multiplicity and the dependence on the specific generators used.
We analyze the correctness of our formulation, establishing sufficient conditions on the samples and the variety.
\item Our methodology takes advantage of the coordinate ring structure to simplify the SDP.
Moreover, it is the first such relaxation independent of Gr\"obner bases.
This makes our methods appealing for many varieties that are easy to sample from but for which Gr\"obner bases computation is intractable.
Examples of such varieties include $SO(n)$, Stiefel manifolds, Grassmannians and secant varieties.
\item We apply for the first time techniques from numerical algebraic geometry to SOS programs.
  In this way, we inherit some of the main strengths from this area.
  We highlight that these methods are trivially parallelizable, since they rely on homotopy continuation of many independent paths. 
  They also allow us to work with straight-line programs (i.e., polynomials do not need to be expanded).
\end{itemize}

\subsection*{Related work}

Polynomial optimization problems have attracted major research in past years.
SDP relaxations based on the SOS/moments method of~\cite{Parrilo2003,Lasserre2001} constitute the major trend of study.
The SOS literature is vast, we refer to~\cite{Lasserre2009,Blekherman2013,laurent2009sums} for an overview.

Our sampling SOS methodology extends the ideas from L\"ofberg and Parrilo in~\cite{Lofberg2004}, where they first consider sampling formulations for unconstrained SOS problems.
They show that sampling formulations offer some numerical advantages over the standard approach.
Most remarkably, the SDP has a low rank structure, which leads to a significant complexity improvement in interior point methods.
In particular, low rank structure is exploited in the solvers SDPT3 and DSDP \cite{sdpt3,dsdp}.
Secondly, the SDP is usually better conditioned, as it relies on a set of orthogonal polynomials instead of a monomial basis.
These properties make sampling formulations appealing, as seen in \cite{Roh2006, Roh2007, Liu2007}.
We will see that these properties are preserved in the variety case considered in this paper.
We remark that our use of the samples differs from \cite{Lofberg2004} in that for us samples carry additional information about the underlying variety~$\V$.

Different methods have been proposed to reduce complexity in SOS programs, in particular by exploiting symmetries, sparsity, and quotient ring structure; see \cite[\S3.3]{Blekherman2013},\cite[\S8]{laurent2009sums} and the references therein.
This paper is only concerned with the last item, but we point out that all these techniques can be combined together.
The Gr\"obner bases method to compute quotient ring SOS certificates was introduced in~\cite{Parrilo2003b}; some further improvements were made in~\cite{Permenter:2012aa}.
This is the default method for several varieties with simple Gr\"obner bases, particularly from combinatorial optimization~\cite[\S8]{laurent2009sums}.
Quotient ring methods have also been used for unconstrained optimization~\cite{Nie2006}.
We point out that there was no ``direct'' method (without computing the radical $\sqrt{I}$) to obtain SOS certificates on the coordinate ring (i.e., SOS($\V$) certificates).

Although the existence (or degree bounds) of SOS certificates is beyond the scope of this paper (see e.g., \cite{Scheiderer2009}), we review some known results for completeness.
In particular, SOS($\V$) certificates exist if:
(i) $\V$ is zero-dimensional,
(ii) $\V$ is one-dimensional and $p$ is both strictly positive and bounded~\cite{Schweighofer2006},
(iii) $\V$ is compact and $p$ is strictly positive~\cite{Schmudgen1991},
(iv) $\V$ is a variety of minimal degree and $p$ is quadratic~\cite{Blekherman2016}.
For most varieties there exist nonnegative polynomials which are not SOS.
Nonetheless, such instances can always be approximated by SOS polynomials (possibly of higher degree) \cite[\S2.6]{Lasserre2009}.

\subsection*{Solution outline}
Our approach to compute $d$-SOS($\V$) certificates follows three main steps.
\begin{enumerate}[(i)]
  \item\label{item:sampling} \textit{Sampling:} 
    Obtain a ``good'' set of samples $Z$ on the variety.
    It will be sufficient for us to consider generic (random) samples on each component of the variety.
  \item\label{item:sdp} \textit{SDP:}
    Given a sample set $Z$, find a sampling pre-certificate $(F,Z)$ using an SDP.
  \item\label{item:verification} \textit{Verification:} 
    Check that the pre-certificate $(F,Z)$ is correct.
    This reduces to the identity testing problem.
\end{enumerate}

The structure of this paper is as follows.
Section~\ref{s:preliminaries} presents some basic algebraic preliminaries.
Afterwards, we approach each of the problems from above, although in a different order to simplify the exposition.
Section~\ref{s:sampling} deals with~\ref{item:sdp}, Section~\ref{s:verify} with~\ref{item:verification}, and Section~\ref{s:irreducible} with~\ref{item:sampling}.
Section~\ref{s:mainprocedure} presents the complete sampling SOS methodology.
Finally, Section~\ref{s:optimization} shows several examples to illustrate our methods.

\section{Preliminaries}\label{s:preliminaries}

\subsection{Algebraic geometry}

Let $\K$ denote a field which is either $\R$ or $\C$, and let $\K[x]=\K[x_1,\ldots,x_n]$ denote the ring of polynomials with coefficients in $\K$.
The \emph{ideal} generated by a set of polynomials $h=\{h_1,\ldots,h_m\}\subset \K[x]$ is
\begin{align*}
I = \langle h\rangle :=\{\textstyle\sum_{i} g_ih_i:g_i\in \K[x]\}.
\end{align*}
The \emph{quotient ring} $\K[x]/I$ is the set of equivalence classes where $f\sim_I g$ if $f-g\in I$.

Given a set of polynomials $h\subset \K[x]$, its complex \emph{algebraic variety} is
\begin{align*}
\V = \V_\C(h) := \{x\in \C^n:h_i(x)=0 \mbox{ for }h_i\in h\}.
\end{align*}
The corresponding real variety is $\V\cap \R^n$.
Note that $\V_\C(h)=\V_\C(\langle h\rangle)$.
In this paper we only consider complex varieties defined by real polynomials.
It is easy to see that a complex variety $\V$ can be defined by real polynomials if and only if it is \emph{self-conjugate}, i.e. its complex conjugate $\conj{\V}$ is itself.

The \emph{coordinate ring} of a variety $\V\subset \C^n$ is the quotient ring
$\K[\V] := \K[x]/\I_\K(\V)$,
where $\I_\K(\V)$ is the vanishing ideal
\begin{align*}
  \I_\K(\V) := \{f\in \K[x]: f(x)=0 \mbox{ for all }x\in \V\}.
\end{align*}
Equivalently, $\K[\V]$ is the set of equivalence classes of polynomials where $f\sim_\V g$ if they define the same function on $\V$.
\begin{rem}
  Hilbert's Nullstellensatz implies that $\I_\K(\V_\C(I))=\sqrt{I}$ for any ideal $I\subset\K[x]$.
  It follows that $\K[\V]$ is equal to the quotient ring $\K[x]/I$ only if $I$ is radical.
\end{rem}

We say that a variety $\V\subset \C^n$ is \emph{irreducible} if it is not the union of two proper varieties.
Note that any variety parametrized by $\C^n$ is irreducible.
An arbitrary variety can be decomposed in a unique way in the form
\begin{align*}
  \V = \V_1 \cup \cdots \cup \V_r,
  \qquad
  \mbox{ where } V_i\not\subseteq V_j \mbox{ for } i\neq j.
\end{align*}
The varieties $\V_i$ are called the \emph{irreducible components} of $\V$.
If $\V$ is self-conjugate, then either $\V_i$ is also self-conjugate, or there is a pair $(\V_i,\V_j)$ of conjugate components.

\subsection{Sampling varieties}\label{s:preliminariessampling}

Our technique requires a sampling oracle for the complex variety $\V$.
More precisely, we need to sample generic (random) points in each irreducible component of $\V$.
Observe that sampling points is easy whenever we have a parametrization.
For instance, we can sample points from $SO(n)$ using the \emph{Cayley parametrization}:
\begin{align}\label{eq:cayley}
  A\mapsto (\id_n -A)(\id_n +A)^{-1}, \quad \mbox{ for skew symmetric } A.
\end{align}
Other parametric varieties include Grassmannians, Stiefel manifolds, secant varieties.

For a general variety~$\V$, a practical way to compute sample points is through the tools of numerical algebraic geometry; 
we refer to \cite{Sommese2005,Bates2013} for an introduction.
Homotopy continuation tools such as Bertini~\cite{bertini} and PHCpack~\cite{phcpack} allow to compute the irreducible decomposition of $\V$, and afterwards to sample an arbitrary number of points in any component.
Typically the most expensive part is to produce the decomposition; sampling points is relatively cheap.
These numerical methods offer the following advantages with respect to symbolic methods such as Gr\"obner bases:
they are trivially parallelizable (each path can be tracked independently),
allow for straight-line programs (polynomials do not need to be to be expanded),
and offer better numerical stability.

\begin{rem}[Complex samples]
  Even though we are only interested in real polynomials, our methods allow the sample points to be complex.
  This is an important feature, since computing real points on a variety is significantly harder than computing complex points.
\end{rem}

\begin{rem}[Zero-dimensional case]
  The results from this paper are most useful for positive-dimensional varieties, particularly if the number of components is relatively small.
  The reason is that we treat each irreducible component separately.
  In particular, we take care of the zero-dimensional part of the variety exhaustively, i.e., we check for all such points that $p(x)$ is indeed nonnegative.
  If the whole variety is zero-dimensional our algorithm reduces to a brute-force search.
\end{rem}

\subsection{SOS certificates on varieties}

Consider a variety $\V$ defined by equations $h = \{h_j\}_j$, and let $I = \langle h\rangle$ be the generated ideal.
There are two traditional SOS methods to certify nonnegativity on $\V\cap\R^n$.
An \emph{equations $d$-SOS} certificate is a tuple of polynomials $(F,g_1,\ldots,g_m)$ such that
\begin{equation}\label{eq:sdprelaxation}
  \begin{aligned}
  &p(x) = F(x) + \sum_j g_j(x)h_j(x); \qquad
  F(x) \mbox{ is SOS}; \qquad
  \deg(F), \deg(g_jh_j)\leq 2d.
  \end{aligned}
\end{equation}
Finding such a certificate reduces to an SDP~\cite{Parrilo2003}.
A \emph{quotient ring $d$-SOS} certificate is a polynomial $F$ such that
\begin{align}\label{eq:quotientring}
  p - F \in I
  \;\;\; ( \mbox{i.e., } p \equiv F \bmod I); \qquad
  F(x) \mbox{ is SOS}; \qquad
  \deg(F)\leq 2d.
\end{align}
Given a Gr\"obner basis of $I$, the above reduces to an SDP~\cite[\S3.3.5]{Blekherman2013}.
For an introduction to Gr\"obner bases and quotient ring computations we refer to~\cite{clo}.

Quotient ring formulations are appealing for two main reasons. 
Firstly, they are \emph{stronger} than equations SOS (i.e., if \eqref{eq:sdprelaxation} is feasible then so is \eqref{eq:quotientring}, but the converse is not true).
And secondly, the size of the associated SDP is \emph{smaller}, not only because of the absence of the equations $g_j$, but since it also takes into account the structure of the quotient ring.
Consequently, quotient ring SOS has become the default approach for varieties with simple Gr\"obner bases (e.g., the hypercube $\{0,1\}^n$).
However, the expense of Gr\"obner bases computation limits its application to further cases.

Our sampling SOS methodology can be seen as a ``better'' quotient ring formulation.
The reason being that we work modulo the \emph{radical} ideal $\sqrt{I}=\I(\V)$, and thus the underlying space is the coordinate ring.
The following diagram illustrates the relations among these three types of certificates.

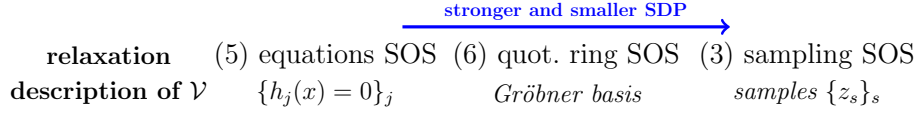
\begin{center}
\begin{tikzpicture}
  \def\x{3.2}
  \def\y{.5}
  \tikzstyle{mybox} = [scale=0.9];
  \tikzstyle{labelbox} = [scale=0.8,font=\bfseries];
  \tikzstyle{smallbox} = [scale=0.8,font=\itshape];
  \tikzstyle{edgelabel} = [scale=0.6,font=\bfseries];
  \tikzstyle{myarrow} = [scale=0.8,blue,very thick,->];
  \node[labelbox] (relax) at (-.9*\x,1*\y) {relaxation};
  \node[labelbox] (desc) at (-.9*\x,0*\y) {description of $\V$};
  \node[mybox] (eqSOS) at (0*\x,1*\y) {\eqref{eq:sdprelaxation} equations SOS};
  \node[smallbox] (desc0) at (0*\x,0*\y) {$\{h_j(x)=0\}_j$};
  \node[mybox] (quotSOS) at (1*\x,1*\y) {\eqref{eq:quotientring} quot.\ ring SOS};
  \node[smallbox] (desc1) at (1*\x,0*\y) {Gr\"obner basis};
  \node[mybox] (coorSOS) at (2*\x,1*\y) {\eqref{eq:sdpsamplesvariety} sampling SOS};
  \node[smallbox] (desc2) at (2*\x,0*\y) {samples $\{z_s\}_s$};
  \draw[myarrow] (.4*\x,2.2*\y) to node[above,edgelabel] {stronger and smaller SDP} (2.1*\x,2.2*\y);
\end{tikzpicture}
\end{center}

\begin{rem}[Hilbert function]\label{thm:hilbertfunction}
  For $k\in \N$, let $\mathcal{L}_k\subset \R[\V]$ be the linear space spanned by the polynomials of degree at most~$k$.
  The function $H_\V(k):= \dim(\mathcal{L}_k)$ is known as the \emph{Hilbert function}, and it plays an important role in sampling SOS (also in quotient ring SOS).
  Indeed, the size of the PSD matrix in the SDP is precisely $H_\V(d)$.
  We will also see in Section~\ref{s:irreducible} that the number of samples we require is given by $H_\V(2d)$.
  The Hilbert function can be bounded as follows~\cite{Chardin1989}:
  \begin{align}\label{eq:hilbertfunction}
    H_\V(k)\leq {n+ k\,\choose k},\quad \mbox{ and } \quad
    H_\V(k)\leq \deg\V \,{\dim\V+ k\,\choose k} \mbox{ if }\V\mbox{ is equidimensional},
  \end{align}
  where $\deg,\dim$ denote the degree and dimension.
  The second bound implies that, for fixed~$\V$, the size of the PSD matrix in sampling SOS is $O(d^{\dim\V})$.
  In contrast, for equations SOS we get~$O(d^n)$.
\end{rem}

\section{Computing pre-certificates}\label{s:sampling}

In this section we show how, given a candidate sample set~$Z$, computing sampling SOS pre-certificates reduces to an SDP.
We will also study what condition do we need on the sample set in order for such pre-certificate to be correct.
The answer will be given by the concept of \emph{poisedness} from polynomial interpolation.
Finally, we will show how to reduce the size of the SDP in order to take advantage of the coordinate ring structure.

\subsection{Sampling SDP}\label{s:computing}

For a degree bound $d$, let $u(x)\in \R[x]^N$ denote the vector with all $N={n+d\choose d}$ monomials of degree at most~$d$.
Recall that a polynomial $F\in \R[x]$ is $d$-SOS if and only if
\begin{align*}
  F(x) = Q\bullet u(x)u(x)^T
\end{align*}
for some positive semidefinite matrix $Q$ (denoted $Q\succeq 0$), where the notation $\bullet$ is for trace inner product~\cite{Parrilo2003}.
Computing a polynomial $F$ satisfying~\eqref{eq:sdpsamplesvariety} reduces to the following SDP:
\begin{equation}\label{eq:SDPpaper}
  \boxed{
\begin{aligned}
  &\mbox{find } &&Q\in \SR^{N},\quad Q\succeq 0 \\
  &\mbox{subject to} &&p(z_s) = Q\bullet u(z_s)u(z_s)^T,&\qquad\qquad \mbox{ for }s=1,\ldots,S
\end{aligned}
}
\end{equation}
where $\SR^N$ denotes the space of $N\times N$ real symmetric matrices.
Note that the matrix $Q$ is real, whereas $p(z_s)$ and $u(z_s)$ are complex.
Thus, each equality imposes a constraint on both the real and the imaginary part, i.e., 
\begin{align*}
&\Re(p(z_s)) = Q\bullet \Re(u(z_s)u(z_s)^T), & \Im(p(z_s)) = Q\bullet \Im(u(z_s)u(z_s)^T ).
\end{align*}

The above SDP has two important features:
the polynomial $p$ can be given as a \emph{straight-line program} (i.e., it does not need to be expanded) and
the constraint matrices have \emph{low rank}. 
Indeed, the rank of the constraint matrices $\Re(u(z_s)u(z_s)^T)$ and $\Im(u(z_s)u(z_s)^T)$ is at most two.
This special rank structure can be exploited in interior point methods, as discussed in~\cite{dsdp,Lofberg2004,Roh2006}.
In particular, the Hessian assembly takes only $O(N^3)$ operations for low rank matrices, as opposed to $O(N^4)$ for unstructured matrices.

Observe that the monomial vector $u(x)$ can be replaced by any other polynomial set with the same linear span.
In particular, we will see in Section~\ref{s:orthogonal} that $u(x)$ can be chosen to be an orthogonal basis with respect to a natural inner product supported on the samples.
Remarkably, this orthogonalization reduces complexity in the SDP by exploiting the algebraic \emph{dependencies} of the coordinate ring $\R[\V]$.
In addition, the conditioning of the problem might improve, as explained in~\cite{Lofberg2004} for the unconstrained case $\V=\C^n$.

\begin{rem}[Kernel/Image form]
  The feasible set of~\eqref{eq:SDPpaper} has the form
    $Q\succeq 0, Q\in \mathcal{Q}$, 
  where 
  \begin{align*}
    \mathcal{Q}=\{Q\in\SR^N:Q\bullet A_i=b_i\}
  \end{align*}
  is an affine subspace.
  We refer to the above representation of $\mathcal{Q}$ as the \emph{kernel form}.
  Alternatively, we can describe $\mathcal{Q}$ explicitly by giving some generators, i.e.,
  \begin{align*}
    \mathcal{Q} = \{Q_0+\sum_j\lambda_jQ_j: \lambda_j\in \R\} 
  \end{align*}
  where $Q_0\bullet A_i = b_i$ and $Q_j\bullet A_i = 0$.
  We refer to this representation as the \emph{image form}.
  Depending on the problem, either of them might be more convenient.
  In particular, if the number of constraints is close to the dimension of $\mathcal{S}^N$ then the latter representation is more compact.
  This will be the case in the applications shown in Sections~\ref{s:procrustes} and~\ref{s:traceratio}.
  For a given problem, we can decide which representation is better by estimating the number of variables used in both of them, as discussed in~\cite{Parrilo2003}.
\end{rem}

\subsection{Poisedness implies correctness}\label{s:samplequotientring}

We just showed how to compute a sampling SOS pre-certificate for a given sample set.
However, this pre-certificate might be incorrect unless we are cautious with the sample set, as illustrated in the next example.
\begin{exmp}[Incorrect pre-certificate]\label{exmp:incorrect}
  Let $\V\subset\C^2$ be the zero set of $h(x):=x_2^2-1$,
  that consists of two complex lines: $\C\times \{1\}$ and $\C\times \{-1\}$.
  Let $p(x):=x_1^2-x_2+1$, which is nonnegative on $\V\cap\R^2$.
  Let $Z:=\{(k,1)\}_{k=1}^S\subset \V$ be a set of samples and let $F(x):=x_1^2$.
  Observe that $(F,Z)$ is a sampling SOS pre-certificate, but it is not correct because $p(0,-1)\neq F(0,-1)$.
  This example illustrates that a sample set, regardless of its size, might lead to incorrect pre-certificates if it does not capture well the geometry of the variety 
  (in this case $Z$ misses one of the components of~$\V$).
\end{exmp}

We now present a condition that guarantees the correctness of a pre-certificate.
Let $\mathcal{R}=\R[\V]$ be the coordinate ring of the variety, which is the space where we will work on.
In particular, we will see the entries of the polynomial vector $u(x)$, as well as $p(x)$, as elements of $\mathcal{R}$.
We need the following definition.

\begin{defn}\label{defn:poised}
  Let $\V\subset \C^n$ be a self-conjugate variety and let $\mathcal{R} = \R[\V]$.
  Let $\mathcal{L}\subset \mathcal{R}$ be a linear subspace and let $Z\subset \V$ be a set of samples.
  We say that $(\mathcal{L},Z)$ is \emph{poised}
  \footnote{In polynomial interpolation it is usually further required that $|Z|=D$, where $D$ is the dimension of $\mathcal{L}$ \cite{Sauer2006}.
We do not impose such condition.} 
  if the only polynomial $q\in \mathcal{L}$ such that $q(z)=0$ for all $z\in Z$ is the zero polynomial.
\end{defn}
\begin{rem}
  For any finite dimensional $\mathcal{L}$ there is a finite set $Z$ such that $(\mathcal{L},Z)$ is poised.
\end{rem}

Let $\mathcal{L}_{d}\subset \mathcal{R}$ be the linear space spanned by the entries of $u(x)$, and let $\mathcal{L}_{2d}\subset \mathcal{R}$ be spanned by the entries of $u(x)u(x)^T$.
Note that $F(x)=Q\bullet u(x)u(x)^T\in \mathcal{L}_{2d}$.
The following proposition tells us that poisedness guarantees the correctness of a sampling SOS pre-certificate.
Thus, a \emph{good set of samples} is one such that $(\mathcal{L}_{2d},Z)$ \emph{is poised}.

\begin{prop}
  Let $\V\subset \C^n$ be a self-conjugate variety, let $\mathcal{R}=\R[\V]$ and let $p\in \mathcal{R}$ be nonnegative on $\V\cap \R^n$.
  Let $(F,Z)$ be a sampling SOS pre-certificate and let $\mathcal{L}_{2d}\subset \mathcal{R}$ be a linear subspace such that $p,F\in \mathcal{L}_{2d}$.
  If $(\mathcal{L}_{2d},Z)$ is poised then $(F,Z)$ is correct.
\end{prop}
\begin{proof}
  Let $g:= p-F \in \mathcal{L}_{2d}$, and observe that $g(z)=0$ for $z\in Z$.
  As $(\mathcal{L}_{2d},Z)$ is poised, this implies that $g=0$ and thus $p=F\in\mathcal{R}$.
\end{proof}

For the rest of this section we assume that the poisedness condition from above is satisfied.
In Section~\ref{s:irreducible} we will discuss how to choose the samples in order to satisfy this requirement.

\subsection{Reducing complexity}\label{s:orthogonal}
The size of the PSD matrix $Q$ from~\eqref{eq:SDPpaper} is ${n+d \choose d}$.
We can reduce the size of this matrix by taking advantage of the coordinate ring structure.
The size of the new matrix will be given by the Hilbert function $H_\V(d)$; see Remark~\ref{thm:hilbertfunction}.
To do so, we simply need to find a basis of the linear subspace $\mathcal{L}_d\subset \mathcal{R}$ spanned by the entries of $u(x)$.
We now explain how to get an orthogonal basis $u^o(x)$ with respect to the inner product given in the next proposition.

\begin{prop}
  Let $\V\subset \C^n$ be a self-conjugate variety and let $\mathcal{R} = \R[\V]$.
  Let $\mathcal{L}_{d}\subset \mathcal{R}$ be a linear subspace and let $Z\subset \V$ be a set of samples.
  Let $\langle \cdot,\cdot\rangle_Z: \mathcal{L}_{d}\times \mathcal{L}_{d}\to \R$ be
\begin{align*}
  \langle f,g \rangle_Z = \sum_{z\in Z} (f(z)g(\conj{z}) + f(\conj{z})g(z)).
\end{align*}
If $(\mathcal{L}_{d},Z)$ is poised then $(\mathcal{L}_{d},\langle \cdot,\cdot\rangle_Z)$ is a real inner product space.
\end{prop}
\begin{proof}
  It is clear that $\langle \cdot, \cdot \rangle_Z$ is bilinear and symmetric.
  Thus, we only need to check positiveness.
  Observe that 
    $\langle f,f \rangle_Z = \sum_{z\in Z} 2|f(z)|^2\geq 0$
, which is zero only if $f(z)=0$ for all $z\in Z$.
As $f\in \mathcal{L}_{d}$, the poisedness condition implies $f=0$.
\end{proof}
\begin{rem}
  Note that if $(\mathcal{L}_{2d},Z)$ is poised then $(\mathcal{L}_{d},Z)$ is also poised.
\end{rem}

To find an orthogonal basis, we will operate on the evaluation matrix $U$ with columns $u(z)$ for $z\in Z$.
Consider the real matrix $W:=[\Re(U)|\Im(U)]$.
Observe that $u(x)$ is an orthogonal basis with respect to $\langle \cdot,\cdot \rangle_Z$ if and only if the rows of $W$ are orthogonal with respect to the standard real inner product.
Thus, we just need to orthogonalize the rows of $W$.
Using an SVD (or rank revealing QR), we can obtain a decomposition $W = TW^o$, where $W^o$ has orthogonal rows and $T$ is a real full rank transformation matrix. 
Let $U^o$ be such that $W^o=[\Re(U^o)|\Im(U^o)]$.
The matrix $U^o$ encodes the new vector of orthogonal polynomials $u^o(x)$.
We note that directly orthogonalizing the matrix $U$ does not work, as the transformation matrix $T$ would be complex.

\begin{algorithm}
  \caption{Orthogonal basis on the coordinate ring}
  \label{alg:orthBasis}
  \begin{algorithmic}[1]
    \Require{Polynomial vector $u(x)$, samples $Z$ of variety $\V$}
    \Ensure{
      Orthogonal basis $u^o(x)$ and its evaluation matrix $U^o$
    }
    \Procedure{OrthBasis}{$u(x),Z$}
    \State $U:=$ evaluation matrix with columns $u(z)$ for $z\in Z$
    \State $W := [\,\Re(U)\,|\,\Im(U)\,]$
    \State orthogonalize $W =: TW^o$, where $W^o(W^o)^T=\id$
    \State let $U^o$ be such that $W^o = [\,\Re(U^o)\,|\,\Im(U^o)\,]$
    \State let $u^o(x)$ be such that $u(x)=Tu^o(x)$
    \State \Return $u^o(x)$, $U^o$
    \EndProcedure
  \end{algorithmic}
\end{algorithm}

\begin{exmp}\label{exmp:SO2_samples}
  Let $\V$ be the complex variety of the set of rotation matrices $SO(2)$, i.e.,
  \begin{align}\label{eq:SO2}
    \V = \{X\in \C^{2\times 2}: X^TX=\id_2, \; \det(X)=1\}.
  \end{align}
  Let $p(X) = 4X_{21}-2X_{11}X_{22}-2X_{12}X_{21}+3$, which is nonnegative on $\V\cap\R^{2\times 2}$.
  We want to find a sampling SOS certificate.
  We can sample points on $\V$ using the Cayley parametrization~\eqref{eq:cayley}.  
  Consider the following 3 complex samples:
  \begin{align*}
z_{1}=\bigl[\begin{smallmatrix}-0.6+0.8i &1.2+0.4i \\ -1.2-0.4i  &-0.6+0.8i \end{smallmatrix}\bigr],\quad
z_{2}=\bigl[\begin{smallmatrix}-1.2+0.4i &0.6+0.8i \\ -0.6-0.8i &-1.2+0.4i \end{smallmatrix}\bigr],\quad
z_{3}=\bigl[\begin{smallmatrix}-0.75+0.25i &0.75+0.25i \\ -0.75-0.25i  &-0.75+0.25i \end{smallmatrix}\bigr].
  \end{align*}

  We fix the degree bound $d=1$, and let $u(x)=(1,X_{11},X_{12},X_{21},X_{22})$ be the monomials of degree at most $d$.
  The matrix of evaluations is:
  \begin{align*}
    U = \left[\begin{smallmatrix}
   1           & 1           & 1           \\
  -0.6 + 0.8i  &-1.2 + 0.4i  &-0.75 + 0.25i\\
  -1.2 - 0.4i  &-0.6 - 0.8i  &-0.75 - 0.25i\\
   1.2 + 0.4i  & 0.6 + 0.8i  & 0.75 + 0.25i\\
  -0.6 + 0.8i  &-1.2 + 0.4i  &-0.75 + 0.25i
\end{smallmatrix}\right]
  \end{align*}
  Using an SVD we obtain the orthogonalized matrix $U^o$ and the corresponding polynomial basis $u^o(x)$.
  Note that $u^o(x)$ has only $3$ elements, as opposed to $u(x)$.

  \begin{minipage}{.5\linewidth}
  \vspace{-18pt}
  \begin{align*}
    U^o = \left[\begin{smallmatrix}
  -0.5955 + 0.1005i  &-0.5955 - 0.1005i  &-0.5201\\
   0.3058 + 0.6116i  &-0.3058 + 0.6116i  &0.2548i\\
  -0.0708 + 0.6411i  &-0.0708 - 0.6411i  &0.4100
\end{smallmatrix}\right]
  \end{align*}
  \end{minipage}
  \begin{minipage}{.5\linewidth}
    {\footnotesize
  \begin{align*}
    u^o(X) = (&X_{21} + X_{22} - .8054, \; X_{21} - X_{22} ,\\
    &X_{21} + X_{22} + 2.4831)
  \end{align*}
  }
  \end{minipage}%

  \vspace{-8pt}
  The sampling SDP is
\begin{eqnarray*}
  \mbox{find }\quad &Q\in \SR^{3},\quad Q\succeq 0 \\
  \mbox{subject to}\quad &p(z_s) = Q\bullet u_su_s^T,&\qquad\qquad \mbox{ for }s=1,2,3
\end{eqnarray*}
  where $u_s$ denotes the $s$-th column of $U^o$.
  Solving the SDP we obtain the sampling pre-certificate $(F,Z)$, where $F(X) = (2X_{21}+1)^2$.
\end{exmp}

\section{Verifying sampling pre-certificates}\label{s:verify}

We now address the problem of testing the correctness of a pre-certificate $(F,Z)$.
This problem is equivalent to determining whether the polynomial $f:=p-F$ is identically zero on the variety~$\V$, and it is known as the \emph{identity testing} problem (see e.g.,~\cite{Saxena2009} and the references therein).
Note that the problem is nontrivial even if $\V=\C^n$, since $f$ can be given as a straight-line program (such as a determinant).
Nonetheless, there is a nice randomized algorithm to solve it, provided that we can efficiently sample the variety.
Recall that generic samples can be obtained as explained in Section~\ref{s:preliminariessampling}.
We now proceed to review the notion of genericity, as well as showing the solution to the identity testing problem.

\subsection{Genericity}\label{s:genericity}

The notion of \emph{genericity} is fundamental in algebraic geometry.
Let $\V\subset\C^n$ be an irreducible variety of positive dimension.
We say that a property hods \emph{generically} on~$\V$ if there is a nonzero polynomial $q\in \C[\V]$ such that the property holds for any $z\in\V$ such that $q(z)\neq 0$.
Informally, this means that the property holds outside of the small bad region given by $q(z)=0$.
In Section~\ref{s:genericity2} we will use a variation of this notion of genericity that is better suited for dealing with real polynomials.

\begin{exmp}
  Let $\V=\C^{m\times m}$ be the space of $m\times m$ matrices.
  The property of being ``nonsingular'' is satisfied generically on~$\V$, since a matrix $A\in\V$ is singular only if $\det(A)=0$.
\end{exmp}

We often say that a sample point $z\in \V$ is \emph{generic} if some property of interest (such as the conclusion of a theorem) is satisfied generically on~$\V$.
For instance, we may say ``a generic $m\times m$ matrix is nonsingular''.
A generic point can be understood as a random point on the variety.

\begin{prop}\label{thm:identitytesting}
  Let $\V$ be an irreducible variety, let $f\in\C[\V]$ be a nonzero polynomial and let $z\in \V$ be a generic sample.
  Then $f(z)$ is nonzero.
\end{prop}
\begin{proof}
  The conclusion holds except in the bad region defined by $f(z)=0$.
\end{proof}

\subsection{Identity testing}

Genericity allows us to derive \emph{randomized} algorithms that succeed with \emph{probability one} with respect to any distribution on $\V$ with full support.
In particular, Proposition~\ref{thm:identitytesting} gives rise to Algorithm~\ref{alg:identitytesting}.
This method efficiently solves the identity testing problem for an irreducible variety (provided that it can be sampled).
Surprisingly, no efficient deterministic algorithm to this problem is known, and it is likely that finding such algorithm is very hard~\cite{Kabanets2004}.

\begin{algorithm}
  \caption{Identity testing over $\C$}
  \label{alg:identitytesting}
  \begin{algorithmic}[1]
    \Require{Polynomial $f\in\C[x]$, irreducible variety $\V$}
    \Ensure{``True'', if $f$ is identically zero on $\V$. 
      ``False'', otherwise.}
    \Procedure{IsZero}{$f,\V$}
    \State $z:=$ generic sample from $\V$
    \State \Return True \algorithmicif\ $f(z)=0$ \algorithmicelse\ False
    \EndProcedure
  \end{algorithmic}
\end{algorithm}

\begin{rem}[Reducible varieties]
  If the variety $\V$ is reducible, we can still solve the identity testing problem provided that we can sample each of its irreducible components.
  We simply need to apply Algorithm~\ref{alg:identitytesting} to each component.
\end{rem}

\begin{rem}[Probability one]
  Randomized algorithms derived from genericity statements provably work with probability one in exact arithmetic.
  However, in floating point arithmetic there is a nonzero probability of error, thus leading to Monte Carlo algorithms; for further discussion see~\cite[\S 4]{Sommese2005}.
\end{rem}

\section{Selecting the samples}\label{s:irreducible}

The missing step to complete our sampling SOS methodology is to describe how to obtain a good set of samples~$Z$.
Recall from Section~\ref{s:samplequotientring} that a good sample set must be such that $(\mathcal{L}_{2d},Z)$ is poised.
Thus the question we address here is the following: 
given a linear space $\mathcal{L}_{2d}\subset \mathcal{R}=\R[\V]$, how can we get a sample set $Z$ such that $(\mathcal{L}_{2d},Z)$ is poised.
We will see in this section that this condition can be satisfied with a generic set of samples.

\subsection{Poisedness again}\label{s:poised}

Before proceeding to the selection of the samples, we first present an alternative characterization of poisedness.
Let $\mathcal{L}\subset\mathcal{R}$ be a finite dimensional subspace.
Let $v(x)\in \mathcal{R}^N$ be a polynomial vector whose entries span $\mathcal{L}$.
Let $U$ be the matrix with columns $v(z)$ for $z\in Z$ and let $\hat{U}:=[U|\conj{U}]$.
We will refer to the (complex) rank of matrix $\hat{U}$ as the \emph{empirical dimension} of $\mathcal{L}$ with respect to $Z$.
It is easy to see that it does not depend on the choice of generators.

\begin{lem}\label{thm:empdim}
  $(\mathcal{L},Z)$ is poised if and only if the dimension of $\mathcal{L}$ is equal to its empirical dimension.
\end{lem}
\begin{proof}
  Let $D:=\dim(\mathcal{L})$ and assume that $v(x)\in \mathcal{R}^D$ is a basis.
  Assume first that the $\rank(\hat{U})=D$.
  Note that any $q\in \mathcal{L}$ can be written uniquely in the form $q(x)=\mu^Tu(x)$ for some $\mu\in \R^D$.
  The condition that $q(z)=0$ for $z\in Z\cup \conj{Z}$ implies that $\mu^T\hat{U}=0$.
  As $\hat{U}$ has full row rank then $\mu=0$, and thus $(\mathcal{L},Z)$ is poised.
  Assume now that $\hat{U}$ does not have full row rank.
  Then there is some nonzero $\lambda\in\C^D$ such that $\lambda^T\hat{U}=0$.
  Observe that this implies that $\Re(\lambda)^TU=\Im(\lambda)^TU=0$, where $\Re,\Im$ denote the real/imaginary part.
  Thus, there is a nonzero $\mu\in \R^D$ such that $\mu^TU=0$.
  Considering the polynomial $q(x):=\mu^Tu(x)$, we conclude that $(\mathcal{L},Z)$ is not poised.
\end{proof}
\begin{rem}\label{thm:poisedsamples}
  Since matrix $\hat{U}$ has $2|Z|$ columns, it follows from the lemma that if $(\mathcal{L},Z)$ is poised then $2|Z|\geq \dim \mathcal{L}$.
\end{rem}

\begin{exmp}
  Let $\V = \C$ and $\mathcal{R} = \R[x]$ be the space of univariate polynomials.
  Let $\mathcal{L}$ be the set of polynomials of degree less than $N$ and let $v(x)=(1,x,\ldots,x^{N-1})$.
  Let $Z\in \C^{N/2}$ be a tuple of complex samples, and let $\hat{Z}$ be the concatenation of $Z$ and $\conj{Z}$.
  The evaluation matrix $\hat{U}$ in this case is the Vandermonde matrix of $\hat{Z}$, which is singular only if there are repeated elements in $\hat{Z}$.
  Therefore, $(\mathcal{L},Z)$ is poised if and only if the elements of $\hat{Z}$ are all distinct.
\end{exmp}

\subsection{How many samples}\label{s:genericity2}

We return to the problem of finding a sample set $Z$ such that $(\mathcal{L}_{2d},Z)$ is poised.
As we decided that the samples will be random, the only missing point is to determine how many samples to take.
Remark~\ref{thm:poisedsamples} tells us that we need at least $\lceil{D/2}\rceil$ samples, where $D:=\dim(\mathcal{L}_{2d})$.
We wonder if this condition is \emph{generically sufficient} to guarantee poisedness.
\begin{ques}
  Let $\V$ be a self-conjugate variety.
  Let $\mathcal{L}_{2d}\subset \R[\V]$ be a $D$-dimensional linear subspace and let $Z\subset \V$ be a generic set of samples with $|Z|\geq D/2$.
  Is $(\mathcal{L}_{2d},Z)$ poised?
\end{ques}

In order to make sense of the above question, we have to be more precise about the meaning of a generic set of samples.
In Section~\ref{s:genericity} we saw the definition of a generic sample of an irreducible variety.
We have to extend this definition to multiple samples, taken possibly from a reducible variety.
Below we formalize the exact notion of genericity we use.
It is slightly different from the one in Section~\ref{s:genericity} as it includes the complex conjugates of the samples.
The reason for including the conjugates is that it reflects the fact that we are working with real polynomials.

\begin{defn}
  Let $\W\subset \C^n$ be an irreducible variety and let $Z=(z_1,\ldots,z_S)$ be a tuple of $S$~samples in~$W$.
  We say that $Z$ satisfies a property \emph{c-generically} (conjugate generically)
  if there is a polynomial $q\in\C[z_1,\ldots,z_S,\conj{z_1},\ldots,\conj{z_S}]$ such that: 
  \begin{itemize}
    \item $q(z_1,\ldots,z_S,\conj{z_1},\ldots,\conj{z_S})$ is not identically zero when $z_1,\ldots,z_S\in \W$.
    \item the property holds whenever $q(z_1,\ldots,z_S,\conj{z_1},\ldots,\conj{z_S})\neq 0$.
  \end{itemize}
  Let $\W_1,\ldots,\W_r$ be irreducible varieties and let $Z_1\subset \W_1,\ldots,Z_r\subset\W_r$ be tuples of samples.
  We say that $(Z_1,\ldots,Z_r)$ satisfies a property c-generically if there are polynomials $q_1\in\C[Z_1,\conj{Z_1}],\ldots,q_r\in\C[Z_r,\conj{Z_r}]$ such that: 
  \begin{itemize}
    \item $q_i(Z_i,\conj{Z_i})$ is not identically zero on $\W_i$, for $1\leq i\leq r$.
    \item the property holds whenever $q_1(Z_1,\conj{Z_1})\neq 0,\ldots,q_r(Z_r,\conj{Z_r})\neq 0$.
  \end{itemize}
  We say that $Z$ (resp. $Z_1,\ldots,Z_r$) is a \emph{c-generic} set of samples if it satisfies some property of interest c-generically.
\end{defn}

In the next section we will show that for an irreducible variety (or a conjugate pair of irreducible varieties) the answer to the question from above is positive.
However, for reducible varieties, we need to make sure that we have enough samples in each irreducible component, as will be discussed in Section~\ref{s:reducible}.
Example~\ref{exmp:incorrect} illustrates what might go wrong if we do not have enough samples in some component.

\subsection{The irreducible case}

Assume now that $\V = \W\cup \conj{\W}$, where $\W\subset \C^n$ is an irreducible variety.
This means that either $\V$ is a self-conjugate irreducible variety, or it is a conjugate pair of irreducible varieties.
In the latter case, note that we can assume without loss of generality that $Z\subset \W$, by possibly exchanging some samples with their complex conjugates.
We show now that if the samples $Z$ are c-generic and are at least as many as the dimensionality of the problem, then the poisedness property is satisfied.

\begin{thm}\label{thm:irreduciblepositive}
  Let $\W\subset \C^n$ be an irreducible variety, let $\V = \W\cup \conj{\W}$ and let $\mathcal{R}=\R[\V]$.
  Let $\mathcal{L}_{2d}\subset \mathcal{R}$ be a linear subspace and let $Z\subset \W$ be a c-generic set of samples.
  If $|Z|\geq D/2$, where $D:=\dim(\mathcal{L}_{2d})$, then $(\mathcal{L}_{2d},Z)$ is poised
  \footnote{This theorem is a special instance of the dimensionality problem in polynomial interpolation, and more elaborate versions can be found in the literature~\cite{Ciliberto:2001aa}.}.
\end{thm}
\begin{proof}
  Let $v(x)\in \mathcal{R}^D$ be a basis of $\mathcal{L}_{2d}$.
  Let $Z_j:=\{z_1,\ldots,z_j\}$, let $V_j\in \C^{D\times j}$ be the matrix with columns $\{v(z)\}_{z\in Z_j}$ and let $\hat{V}_j:=\left[\Re(V_j)|\Im(V_j)\right]\in \R^{D\times 2j}$. 
  Also denote $W_j := [\hat{V}_j| \Im(v(z_{j+1}))]\in \R^{D\times 2j+1}$.
  Because of Lemma~\ref{thm:empdim}, we just need to show that the matrix $\hat{V}_S$ has rank $D$.
  To this end, we will show the following statements:
  \begin{itemize}
    \item if $\hat{V}_{j-1}$ is full rank then $W_{j-1}$ is full rank c-generically.
    \item if $W_{j-1}$ is full rank then $\hat{V}_{j}$ is full rank c-generically.
  \end{itemize}
  Clearly these statements imply that $\hat{V}_S$ is full rank.
  Given the similarity between the two of them, we only prove the latter.

  Let $j\leq D/2$ and assume that $W_{j-1}$ is full rank.
  We will show that there is a polynomial $Q\in\C[Z_j,\conj{Z_j}]$ which is not identically zero on $\W$, and such that $\hat{V}_j$ is full rank whenever $Q(Z_j,\conj{Z_j})\neq 0$.

  Assume that $\hat{V}_{j}$ is not full rank.
  Then there must exist a vector $\lambda\in \R^{2j-1}$ such that 
  $$v(z_{j})+ v(\conj{z_{j}}) = 2\Re(v(z_{j})) = W_{j-1}\lambda.$$
  As $W_{j-1}$ has less than $D$ columns, there is some nonzero vector $\mu\in \R^D$ in its left kernel.
  Note that $\mu=\mu(Z_j,\conj{Z_j})$ can be parametrized as a rational function of $Z_j,\conj{Z_j}$, given that $W_{j-1}$ is full rank.
  Let $q_\mu(x) := \mu^T v(x)\in \mathcal{R}$, which is nonzero due to the linear independence of $v(x)$.
  Observe that
  \begin{align*}
   q_\mu(z_{j}) + q_\mu(\conj{z_{j}}) 
   = \mu^TW_{j-1}\,\lambda
  = 0.
  \end{align*}
  As the coefficients of $q_\mu$ are rational functions on $Z_j,\conj{Z_j}$, we conclude that the samples satisfy a nonzero algebraic equation $Q\in\C[Z_{j},\conj{Z_{j}}]$.
\end{proof}
\begin{rem}
  If the samples are real, it can be shown in a similar way that we need $|Z|\geq D$.
\end{rem}

\subsection{Verifying the number of samples}

We just showed that, under genericity assumptions, the poisedness property is satisfied whenever we have as many samples as the dimension of the space.
Concretely, we need to have $\lceil D/2\rceil$ complex samples, where $D=\dim(\mathcal{L}_{2d})$.
However, as the dimension $D$ is not known a priori, it is uncertain how many samples to take.
Therefore, we need some way to estimate such dimension, and the natural quantity to consider is the empirical dimension $D_{e}$.
The following corollary gives us a simple test that guarantees that $D = D_e$.

\begin{cor}\label{thm:sufficientsamples}
  Let $\W\subset \C^n$ be an irreducible variety, let $\V = \W\cup \conj{\W}$ and let $\mathcal{R}=\R[\V]$.
  Let $\mathcal{L}_{2d}\subset \mathcal{R}$ be a linear subspace and let $Z\subset \W$ be a c-generic set of samples.
  Let $D$ be the dimension of $\mathcal{L}_{2d}$ and let $D_e$ be its empirical dimension with respect to $Z$.
  If $D_e<2|Z|$ then $(\mathcal{L}_{2d},Z)$ is poised (i.e., $D = D_e$).
\end{cor}
\begin{proof}
  If $2|Z| < D$ it follows from the proof of Theorem~\ref{thm:irreduciblepositive} that $D_e = 2|Z|$.
  Therefore, we must have that $2|Z| \geq  D$, and thus $(\mathcal{L}_{2d},Z)$ is poised because of Theorem~\ref{thm:irreduciblepositive}.
\end{proof}

\begin{algorithm}
  \caption{Test samples}
  \label{alg:testSamples}
  \begin{algorithmic}[1]
    \Require{Polynomial vector $u(x)$, samples $Z$ of a variety $\V$}
    \Ensure{``True'', if generically we must have that $(\mathcal{L}_{2d},Z)$ is poised, where $\mathcal{L}_{2d}\subset \R[\V]$ is spanned by $u(x)u(x)^T$. 
      ``False'', if we cannot guarantee it.}
    \Procedure{GoodSamples}{$u(x),Z$}
    \State $\hat{U}_2 :=$ matrix with columns $\mathrm{vec}(u(z)u(z)^T)$, for $z\in Z\cup \conj{Z}$
    \State \Return False \algorithmicif\ $\hat{U}_2$ has full column rank \algorithmicelse\ True
    \EndProcedure
  \end{algorithmic}
\end{algorithm}

The above corollary suggests a simple strategy that is summarized in Algorithm~\ref{alg:testSamples}.
We form the vector $u_2(x) = \mathrm{vec}(u(x)u(x)^T)$, whose entries span $\mathcal{L}_{2d}$.
Then we build the matrix of evaluations $\hat{U}_2$ with columns $u_2(z)$ for $z\in Z\cup \conj{Z}$.
The rank of this matrix is the empirical dimension~$D_e$.
If $\hat{U}_2$ does not have full column rank the above corollary holds.

\begin{rem}
  Consider the Hermitian matrix $\hat{U}_2^*\hat{U}_2$, where $^*$ denotes the conjugate transpose.
  This matrix is often much smaller than $\hat{U}_2$, and it can be constructed efficiently as
  \begin{align*}
    \hat{U}_2^*\hat{U}_2 = [\langle u(z_i),u(z_j)\rangle^2]_{z_i,z_j \in Z\cup \conj{Z}} = (\hat{U}^*\hat{U})\circ(\hat{U}^*\hat{U})
  \end{align*}
  where $\circ$ denotes the Hadamard product.
  Therefore, it is practical to use matrix $\hat{U}_2^*\hat{U}_2$ instead of $\hat{U}_2$, given that they have the same rank.
\end{rem}

\begin{exmp}
  Consider the case of Example~\ref{exmp:SO2_samples}.
  We used $S=3$ samples to compute the pre-certificate.
  To verify that the number of samples was sufficient, we construct the matrix
  \begin{align*}
    \hat{U}_2^*\hat{U}_2 = \left[\begin{smallmatrix}
   1.5581            &-0.2937 + 0.2562i  & 0.1730 - 0.1158i  & 0.0902 + 0.1118i  & 0.0981            &-0.0676 - 0.0720i\\
  -0.2937 - 0.2562i  & 1.5581            & 0.1730 + 0.1158i  & 0.0981            & 0.0902 - 0.1118i  &-0.0676 + 0.0720i\\
   0.1730 + 0.1158i  & 0.1730 - 0.1158i  & 0.2535            &-0.0676 - 0.0720i  &-0.0676 + 0.0720i  & 0.1396\\
   0.0902 - 0.1118i  & 0.0981            &-0.0676 + 0.0720i  & 1.5581            &-0.2937 - 0.2562i  & 0.1730 + 0.1158i\\
   0.0981            & 0.0902 + 0.1118i  &-0.0676 - 0.0720i  &-0.2937 + 0.2562i  & 1.5581            & 0.1730 - 0.1158i\\
  -0.0676 + 0.0720i  &-0.0676 - 0.0720i  & 0.1396            & 0.1730 - 0.1158i  & 0.1730 + 0.1158i  & 0.2535 
\end{smallmatrix}\right]
  \end{align*}
  The rank of this matrix is $5$, and thus the condition from Corollary~\ref{thm:sufficientsamples} is satisfied.
  Therefore, the number of samples is sufficient.
\end{exmp}

\subsection{Reducible varieties}\label{s:reducible}

The analysis made so far makes an irreducibility assumption on the variety $\V$.
This assumption is satisfied for many varieties, in particular for any variety parametrized by $\C^n$.
Even if $\V$ is not irreducible, we can always work with each of its irreducible components independently.
Indeed, note that $p\geq 0$ on some variety if and only if $p\geq 0$ on each irreducible component.

Nonetheless, there are circumstances in which we may not want to impose an irreducibility assumption.
For example, if the variety has bad numerical properties and thus its irreducible components cannot be accurately estimated.
In such situations, we can repeat the same analysis from before if we have some method that samples points from each irreducible component.
For instance, if we intersect the variety $\V$ with a generic hyperplane of complementary dimension, the intersection is a finite set that contains points in each irreducible component.
Note that we do not know which component do the samples belong to, but we are certain that there is at least one sample in each component.

The following corollary shows that if we have a sample set with enough points on each irreducible component, then $(\mathcal{L}_{2d},Z)$ is poised.

\begin{cor}
  Let $\W\subset\C^n$ be a variety, let $\V = \W\cup\conj{\W}$ and let $\mathcal{R}=\R[\V]$.
  Let $\mathcal{L}_{2d}\subset \mathcal{R}$ be a linear subspace.
  Let $\W = \W_1\cup\cdots\cup\W_r$ be the irreducible decomposition, and let $Z_1\subset \W_1, \ldots,Z_r\subset \W_r$ be c-generic sets of samples.
  If $|Z_i|\geq D/2$ for all $i$, where $D:=\dim(\mathcal{L}_{2d})$, then $(\mathcal{L}_{2d},Z)$ is poised.
\end{cor}
\begin{proof}
  Let $f\in \mathcal{L}_{2d}$ be such that $f(z)=0$ for all $z\in Z$.
  We want to show that $f$ is the zero polynomial in $\C[\V]$.
  Let $\V_i:=\W_i\cup \conj{\W_i}$ and let $\psi_i: \C[\V]\to \C[\V_i]$ be the restriction operator.
  It is clear that the dimension of $\psi_i(\mathcal{L}_{2d})$ is at most $D$.
  Thus, Theorem~\ref{thm:irreduciblepositive} says that $(\psi_i(\mathcal{L}_{2d}),Z_i)$ is poised whenever $q_i(Z_i,\conj{Z_i})\neq 0$, for some polynomial $q_i$ which is nonzero on $\W_i$.
  Note that $\psi_i(f)$ evaluates to zero on $Z_i$, and thus $\psi_i(f)$ must be the zero element in $\C[\V_i]$ whenever $q_i(Z_i,\conj{Z_i})\neq 0$.
  Finally, observe that $(\psi_1\times\cdots\times \psi_k):\C[\V]\to \C[\V_1]\times \cdots \times \C[\V_k]$ is injective.
  We conclude that whenever $q_i(Z_i,\conj{Z_i})\neq 0$ then $(\psi_1\times\cdots\times \psi_k)(f)$ is zero and thus $f$ must be zero.
\end{proof}

\section{Computing sampling certificates}\label{s:mainprocedure}

We already developed all the tools needed to find a sampling SOS certificate, and we now put them together.
Algorithm~\ref{alg:mainprocedure} summarizes our method for the case of an irreducible variety~$\V$.
Naturally, the most computationally expensive part is solving the SDP.
Recall from Theorem~\ref{thm:irreduciblepositive} that the number of samples required is
\begin{align}\label{eq:nsamples}
  S_{\min} := \lceil H_{\hat{\V}}(2d)/2\rceil,\qquad
  H_{\hat{\V}}(2d):=\dim(\mathcal{L}_{2d}) \leq 
  \min\left\{
  {n+2d\choose 2d},\;
  \deg \hat{V} {\dim \hat{V} + 2d \choose 2d}
  \right\},
\end{align}
where $\hat{\V} = \V\cup\conj{\V}$, and where we used the bound from~\eqref{eq:hilbertfunction}.
Since $H_{\hat{\V}}(2d)$ is unknown in general, we use a simple search strategy in Algorithm~\ref{alg:mainprocedure}.
The algorithm terminates when the number of samples is at least~$S_{\min}$.

In the case of a reducible variety, we might use Algorithm~\ref{alg:mainprocedure} for each of its irreducible components separately.
If we cannot reliably identify such components we need to take into account the considerations from Section~\ref{s:reducible}.

\begin{rem}[Zero-dimensional case]
  Note that a zero-dimensional variety is reducible, each component consisting of a single point.
  Thus, in such case our algorithm reduces to a brute-force enumeration over all solutions, and better strategies may exist.
  The main problem to address is that of producing small poised sets.
  We leave this as an open problem.
\end{rem}

\begin{algorithm}
  \caption{Sampling SOS}
  \label{alg:mainprocedure}
  \begin{algorithmic}[1]
    \Require{
      Polynomial $p\in\R[x]$ (given by an evaluation oracle),
      irreducible variety $\V\subset\C^n$ (given by a sampling oracle),
      degree bound $d\in\N$
    }
    \Ensure{$d$-SOS($\V$) certificate $F$, if it exists. 
      ``Null'', if no certificate exists.}
    \Procedure{SamplingSOS}{$p,\,\V,\,d$}
    \State $u(x):=$ vector with all monomials up to degree $d$
    \State $S:=$ initial guess on the number of samples (an upper bound is given in~\eqref{eq:nsamples})
    \State\label{line:samples}  $Z:=$ generic set of $S$ samples from $\V$
    \State $u(x):=$ \Call{OrthBasis}{$u(x),Z$} \Comment{find basis of $\R[\V]$}
    \If {\NOT\Call{GoodSamples}{$u(x),Z$}} \Comment{check if there are enough samples}
    \State increase $S$ and \GOTO \ref{line:samples}
    \EndIf
    \State $Q:=$ solution of SDP~\eqref{eq:SDPpaper}\; (\algorithmicif\ none \Return Null)
    \Comment{solve SDP}
    \State $F(x):= Q \bullet u(x)u(x)^T$
    \If {\NOT\Call{IsZero}{$p-F,\,\V$}} \Comment{verify correctness}
    \State the sample set was not generic enough; \GOTO \ref{line:samples}
    \EndIf
    \State \Return $F$
    \EndProcedure
  \end{algorithmic}
\end{algorithm}

\section{Examples}\label{s:optimization}

We now show several examples and numerical evaluations to illustrate our methodology.
We implemented our algorithms in Matlab, using SDPT3~\cite{sdpt3} to solve the semidefinite programs.
We also use Macaulay2~\cite{macaulay2} for Gr\"obner bases, and Bertini~\cite{bertini} for numerical algebraic geometry computations.
The experiments are performed on an i7 PC with $8$ cores of $3.40$ GHz, $15.6$ GB RAM, running Ubuntu $14.04$. 

We will compare our techniques with the following two methods: equations SOS \eqref{eq:sdprelaxation} and the (Gr\"obner bases based) quotient ring SOS \eqref{eq:quotientring}.
For the convenience of the reader unfamiliar with these methods, Appendix~\ref{s:sosvariety} illustrates them with concrete, simple examples.

\begin{rem}[Polynomial optimization]
Since some of the examples below are polynomial optimization problems, we recall that there are well studied SOS relaxations~\cite{Blekherman2013,laurent2009sums,Lasserre2009}.
Indeed, minimizing a polynomial $p(x)$ is equivalent to finding the largest $\gamma$ such that $p(x) - \gamma$ is nonnegative, which can be relaxed to be a $d$-SOS polynomial.
The solution $\gamma^*$ obtained with the SDP gives a valid lower bound on $p(x)$.
SOS lower bounds tend to be very good in practice, being the true minimum in certain applications~\cite{Blekherman2013,Lasserre2009}.
Moreover, if the minimizer $x^*$ is unique and the dual matrix is rank one, then $x^*$ might be recovered from the dual variables.
In particular, 
$x^* = \Re(\textstyle\sum_s y_s^* z_s)$ 
for the sampling SDP in~\eqref{eq:SDPpaper}, where $y_s^*$ are the (complex) dual variables of the equality constraints.
\end{rem}

\subsection{Nilpotent matrices}

Let $\V := \{X\in \C^{n\times n}: X^n=0\}$ be the variety of nilpotent matrices.
Let $p(X):= \det(X + \id_n)$, which is nonnegative on $\V$ (moreover, it is identically one).
We compare different SOS methodologies to certify this.

First, consider the sampling approach.
Let the degree bound $d=1$, and let us take $S={n+2\choose 2}$ samples, which are always sufficient.
Note that it is very easy to sample nilpotent matrices.
For instance, we can generate a random triangular matrix with zero diagonal, and then apply a similarity transformation.
For each sample $X_s\in \V$, we can efficiently evaluate $p(X_s)$ with Gaussian elimination.
As $p(X_s)=1$ for all samples $X_s$, we will obtain the trivial SOS decomposition $p(X)\equiv_\V (1)^2$.

Consider now the Gr\"obner bases approach.
Let $h\subset \R[X]$ be the $n^2$ equations given by $X^n=0$.
We want to compute a Gr\"obner basis of $h$.
Note, however, that the total number of terms in $h$ is on the order of $n^{n+1}$, and the polynomials are all of degree $n$.
Therefore, this Gr\"obner basis computation is extremely complicated.

If we are smarter, we can take a different set of defining equations of $\V$.
Consider the polynomial $Q_X(t):= \det(t\,\id_n - X) - t^n$, and let $h'\subset \R[X]$ be the equations given by the coefficients of $Q_X(t)$.
It turns out that $h'$ generates the radical ideal of $\langle h\rangle$, and moreover, it is a Gr\"obner basis \cite[\S 7]{Jantzen2004}.
However, $h'$ has more than $n!$ terms. 
Once we have the Gr\"obner basis $h'$, we need to compute the normal form of $p$.
To obtain this normal form we need to consider $p$ as a dense polynomial.
As both $p$ and $h'$ have on the order of $n!$ terms, performing this reduction is computationally intractable.
If we are able to reduce it, we will conclude that $p(X)\equiv_\V 1$, as before.

Finally, note that equations SOS suffers from the same problems of the Gr\"obner bases approach.
For this method there is an additional problem, which is that the monomial basis $u(X)$ will be very large in order to account for all the monomials in $p(X)$ and $h(X)$.
This problem was avoided in the previous methods because of the quotient ring reductions.

This example illustrates two of the advantages of the sampling formulation:
it avoids the algebraic problem of deciding which equations to use (e.g., $h$ vs. $h'$), and it allows the use of straight-line programs (e.g., Gaussian elimination) for more efficient evaluations.

\subsection{Weighted orthogonal Procrustes}\label{s:procrustes}
We consider a family of optimization problems over varieties of orthogonal matrices.
The Stiefel manifold $\mathit{St}(k,\R^n)$ is the set of orthonormal $k$-frames in $\R^n$.
We identify it with the set of matrices $X\in \R^{n\times k}$ such that $X^TX=\id_k$.
Note that we can easily sample points from this variety, for instance, by using the Cayley parametrization.
Alternatively, we can orthogonalize a random real matrix.

The weighted orthogonal Procrustes problem, also known as Penrose regression problem, asks for a matrix $X\in \mathit{St}(k,\R^n)$ that minimizes $\|AXC-B\|$, for some matrices $A\in\R^{m_1\times n}$, $B\in \R^{m_1\times m_2}$, $C\in\R^{k\times m_2}$.
There is no closed form solution for this problem, and several local optima may exist~\cite{Viklands:2006aa,Chu2001}.

Let $u(x)$ consist of all monomials up to some degree bound~$d$.
The sampling SDP is:
\begin{equation*}
\begin{aligned}
  &\max_{\gamma\in \R,\,Q\succeq 0} &&\gamma \\
  &\mbox{subject to} &&\|AX_sC-B\|^2 - \gamma = Q\bullet u(X_s)u(X_s)^T,&\qquad \mbox{ for }s=1,\ldots,S \\
  & &&X_s \in \mathit{St}(k,\R^n)
\end{aligned}
\end{equation*}

\begin{exmp}[\cite{Chu2001}, Ex~2]
  Let $(n,k,m_1,m_2)=(4,3,5,3)$ and consider the matrices
  \begin{align*}
A^T=\left[\begin{smallmatrix}
    0.2190 &0.0470 &0.6789 &0.6793 &0.9347\\
    0.3835 &0.5194 &0.8310 &0.0346 &0.0535\\
    0.5297 &0.6711 &0.0077 &0.3834 &0.0668\\
    0.4175 &0.6868 &0.5890 &0.9304 &0.8462
  \end{smallmatrix}\right], \quad
B^T=\left[\begin{smallmatrix}
    0.6526 &0.2110 & 0.2229 &-0.4104 &-0.9381\\
    0.6942 &0.2204 & 0.2015 & 0.2994 & 1.0943\\
    0.8299 &1.1734 &-0.1727 & 0.0474 &-0.2351
  \end{smallmatrix}\right], \quad
  C = \id_3.
  \end{align*}
  We consider the degree~$1$ SOS relaxation.
  Following Algorithm~\ref{alg:mainprocedure}, we find out that $S = 43$ complex (or $85$ real) samples are sufficient.
  More generally, the required number of samples is a half of $H_\V(2) = {nk+2\choose 2}-{k+1\choose 2}$.
  By solving the above SDP we obtain a lower bound of $1.118147$ on the minimum norm $\|AXC-B\|$.
  Furthermore, the dual SDP matrix has rank one, and thus we can recover a solution achieving such lower bound:
  \begin{align*}
(X^*)^T=\left[\begin{smallmatrix}
   -0.0895 & 0.7472 &0.2732 &-0.5992\\
    0.7726 &-0.1843 &0.6035 &-0.0702\\
   -0.5277 & 0.0163 &0.7309 & 0.4324
  \end{smallmatrix}\right].
  \end{align*}
\end{exmp}

Table~\ref{tab:stiefel} compares different SDP formulations of the degree~$1$ SOS relaxation of the weighted orthogonal Procrustes problem.
We consider the case where $m_1=n$ and $m_2=k$.
The table shows the number of variables/constraints and the computation time for the equations SDP and the sampling SDP.
The computation is performed on random instances, in which matrices $A$, $B$, $C$ are generated from the standard normal distribution.
For the sampling SDP we use the image form of the SDP (see Section~\ref{s:computing}), given that it has low codimension.
We remark that for the sampling SDP we include the preprocessing time, i.e., Algorithms~\ref{alg:orthBasis} and~\ref{alg:testSamples}.

We point out that the better performance of sampling SDP is due to the fact that it makes use of the quotient ring structure.
Although a similar sized SDP could be derived using Gr\"obner bases, Table~\ref{tab:stiefel} shows that Gr\"obner bases computation is very expensive, much more than solving the (larger) equations SDP.
In particular, Macaulay2 ran out of memory for $n=7$, $k=5$.

\begin{table}[htbp]
  \centering
  \caption{Degree~$1$ SOS relaxations for the weighted orthogonal Procrustes problem}
    \begin{tabular}{cc||ccc||ccc||c}
    \toprule
    \multirow{2}[0]{*}{$n$} & \multirow{2}[0]{*}{$k$} & \multicolumn{3}{c||}{Equations SDP} & \multicolumn{3}{c||}{Sampling SDP} & {Gr\"obner bases} \\
          &       & variables & constraints & time$(s)$ & variables & constraints & time$(s)$ & time$(s)$ \\
    \midrule
    4     & 2     & 178   & 73    & 0.52   & 46    & 42    & 0.10 & 0.00 \\
    5     & 3     & 682   & 233   & 0.65   & 137   & 130   & 0.11 & 0.03 \\
    6     & 4     & 1970  & 576   & 1.18   & 326   & 315   & 0.15 & 9.94 \\
    7     & 5     & 4727  & 1207  & 3.56   & 667   & 651   & 0.31 & out of mem.\\
    8     & 6     & 9954  & 2255  & 13.88  & 1226  & 1204  & 0.70 & out of mem.\\
    9     & 7     & 19028 & 3873  & 42.14  & 2081  & 2052  & 2.11 & out of mem.\\
    10    & 8     & 33762 & 6238  & 124.43 & 3322  & 3285  & 5.07 & out of mem.\\
    \bottomrule
    \end{tabular}%
  \label{tab:stiefel}%
\end{table}%

\subsection{Trace ratio problem}\label{s:traceratio}
We now consider a problem on the Grassmaniann manifold $\mathit{Gr}(k,\R^n)$, which is the set of all $k$-dimensional subspaces of $\R^n$.
Note that we can easily sample points on $\mathit{Gr}(k,\R^n)$ by considering the subspace spanned by $k$ random vectors.
By identifying a subspace with the orthogonal projection onto it, we can view $\mathit{Gr}(k,\R^n)$ as the set of matrices $X\in \SR^{n}$ satisfying $X^2=X$ and $\trace(X)=k$; so this is indeed a variety.
The trace ratio problem looks for the maximizer of $\frac{\trace(AX)}{\trace(BX)}$ on $\mathit{Gr}(k,\R^n)$,
for some given matrices $A,B\in \SR^{n}$, $B\succ 0$.
This problem arises in machine learning, and it can be efficiently solved by iterative methods, given that it has a unique local maximum~\cite{Wang:2007aa}.
We consider the following variation:
\begin{align*}
  \max_{X\in \mathit{Gr}(k,\R^n)} \frac{\trace(AX)}{\trace(BX)} + \trace(CX)
\end{align*}
for some $A,B,C\in \SR^{n}$, $B\succ 0$.
This problem may have several local maxima and thus local methods may not converge to the global optimum \cite{Zhang2014,Zhang2014a}.

To obtain an SOS relaxation, note that the problem is equivalent to minimizing $\gamma$ such that $\trace(BX)(\gamma - \trace(CX)) - \trace(AX)$ is nonnegative on $\mathit{Gr}(k,\R^n)$.
Thus, the SDP to consider is:
\begin{equation*}
\begin{aligned}
  &\min_{\gamma\in \R,\,Q\succeq 0} &&\gamma \\
  &\mbox{subject to} &&\trace(BX_s)(\gamma - \trace(CX_s)) - \trace(AX_s) = Q\bullet u(X_s)u(X_s)^T,&\quad \mbox{ for }s=1,\ldots,S \\
  & &&X_s \in \mathit{Gr}(k,\R^n)
\end{aligned}
\end{equation*}

\begin{exmp}[\cite{Zhang2014a}, Ex~3.1]
  Let $n=3,k=2$ and consider the matrices $A,B,C$ from below.
  For the degree bound $d=1$, Algorithm~\ref{alg:mainprocedure} gives that $S=8$ complex (or 15 real) samples are sufficient.
  In general, the number of samples is a half of $H_\V(2)= {\frac{1}{2}(n^2+n)\choose 2}$.
  Solving the above SDP we get an upper bound of $28.692472$.
  As the dual matrix has rank one, we can recover the optimal solution $X^*$.
  \begin{align*}
A=\left[\begin{smallmatrix}
    11 & 5 &8\\
     5 &10 &9\\
     8 & 9 &5
  \end{smallmatrix}\right], \quad
B=\left[\begin{smallmatrix}
     7 & 7 &7\\
     7 &10 &8\\
     7 & 8 &8
  \end{smallmatrix}\right], \quad
C=\left[\begin{smallmatrix}
    15 &10 &9\\
    10 & 7 &6\\
     9 & 6 &6
  \end{smallmatrix}\right], \quad
X^*=\left[\begin{smallmatrix}
    0.61574 & 0.15424 & 0.46132\\
    0.15424 & 0.93809 &-0.18517\\
    0.46132 &-0.18517 & 0.44617
  \end{smallmatrix}\right]
  \end{align*}
\end{exmp}

As before, we compare the equations SDP and the sampling SDP of the degree~$1$ SOS relaxation.
Table~\ref{tab:grassman} shows the number of variables/constraints and the computation time on random instances for both methods.
It also shows the computation time of Gr\"obner bases.

\begin{table}[htbp]
  \centering
  \caption{Degree~$1$ SOS relaxations for the trace ratio problem}
    \begin{tabular}{cc||ccc||ccc||c}
    \toprule
    \multirow{2}[0]{*}{$n$} & \multirow{2}[0]{*}{$k$} & \multicolumn{3}{c||}{Equations SDP} & \multicolumn{3}{c||}{Sampling SDP} & {Gr\"obner bases} \\
          &       & variables & constraints & time$(s)$ & variables & constraints & time$(s)$ & time$(s)$ \\
    \midrule
    4     & 2     & 342   & 188   & 0.47  & 56    & 45    & 0.10 & 0.00   \\
    5     & 3     & 897   & 393   & 0.71  & 121   & 105   & 0.11 & 0.02   \\
    6     & 4     & 2062  & 738   & 1.34  & 232   & 210   & 0.15 & 0.20   \\
    7     & 5     & 4265  & 1277  & 3.62  & 407   & 378   & 0.19 & 6.04   \\
    8     & 6     & 8106  & 2073  & 9.06  & 667   & 630   & 0.34 & 488.17 \\
    9     & 7     & 14387 & 3198  & 23.83 & 1036  & 990   & 0.61 & out of mem.\\
    10    & 8     & 24142 & 4733  & 58.17 & 1541  & 1485  & 1.18 & out of mem.\\
    \bottomrule
    \end{tabular}%
  \label{tab:grassman}%
\end{table}%

\subsection{Low rank approximation}

Consider the problem of finding the nearest rank~$k$ tensor.
Let $\C^{n_1\times\cdots\times n_\ell}$ denote the set of tensors of order $\ell$ and dimensions $(n_1,\ldots,n_\ell)$ and let $\C^{n_1\times\cdots\times n_\ell}_{\,\leq k}$ be the closure of the space of tensors of rank at most~$k$.
Note that we can easily generate generic samples of rank~$k$ tensors.
Given a real tensor $T\in \R^{n_1\times\cdots\times n_\ell}$, the rank~$k$ approximation problem asks for the nearest point $X\in \R^{n_1\times\cdots\times n_\ell}_{\,\leq k}$, 
i.e., the minimizer of $\|T-X\|^2$ where $\|\cdot\|$ denotes the norm of the vectorization.

Let $d:=\lfloor k/2 \rfloor +1$ and let $u(X)$ be the vector with all monomials of degree at most $d$.
Denoting $\varsigma(X):=\|X\|^2$, 
we consider the following SDP relaxation:
\begin{equation*}
\begin{aligned}
  &\max_{\gamma\in \R,\,Q\succeq 0} &&\gamma \\
  &\mbox{subject to} &&(\|T-X_s\|^2- \gamma)\,\varsigma(X_s)^{d-1} = Q\bullet u(X_s)u(X_s)^T,&\qquad \mbox{ for }s=1,\ldots,S \\
  & &&X_s \in \C^{n_1\times\cdots\times n_\ell}_{\,\leq k}
\end{aligned}
\end{equation*}

We remark that computing the defining equations of the variety $\C^{n_1\times\cdots\times n_\ell}_{\,\leq k}$ is very complicated~\cite{Landsberg:2007aa}.
This means that using traditional SOS methods is usually not possible.

\begin{exmp}[\cite{De-Lathauwer:2000aa}, Ex 3]
  Let $T\in \R^{2\times 2\times 2\times 2}$ be the tensor whose nonzero entries are
  \begin{align*}
    T_{1111} = 25.1,\quad
    T_{1121} = 0.3, \quad
    T_{1212} = 25.6,\quad
    T_{2111} = 0.3, \quad
    T_{2121} = 24.8,\quad
    T_{2222} = 23.
  \end{align*}
  Consider the rank one approximation problem.
  Solving the above SDP ($d=1,S=49$) we obtain the lower bound $42.1216$ on the minimum distance $\|T-X\|$.
  From the dual solution we recover the minimizer $X^*$, whose only nonzero entry is
    $X^*_{1212} = 25.6$.

  Consider now the rank three approximation problem.
  The above SDP ($d=2, S=2422$) gives a lower bound of $23.0000$. 
  Again, we can recover the minimizer $X^*$, whose nonzero entries are
  \begin{align*}
    X^*_{1111} = 25.1,\quad
    X^*_{1121} = 0.3, \quad
    X^*_{1212} = 25.6,\quad
    X^*_{2111} = 0.3, \quad
    X^*_{2121} = 24.8
  \end{align*}
  To see that $X^*$ is rank three, note that after removing the entry $25.6$ we are left with a $2\times 2$ matrix.
\end{exmp}

\subsection{Certifying infeasibility}
Given a complex variety $\V\subset \C^n$ consider the problem of certifying that $\V\cap\R^n$ is empty.
A \emph{Positivstellensatz} infeasibility certificate consists in showing that the constant polynomial $-1$ is SOS on the variety $\V$ \cite{Parrilo2003}.
For instance, if $\V=\{i,-i\}\subset \C$, a Positivstellensatz certificate is that $-1 = x^2$ on the variety $\V$.
We take an approach of numerical algebraic geometry, where we first compute a numerical irreducible decomposition of~$\V$, and then use sampling SOS to obtain the infeasibility certificate.
For a given vector $u(x)$ the SDP problem to solve is:
\begin{equation*}
\begin{aligned}
  &\mbox{find} && Q\succeq 0 \\
  &\mbox{subject to} &&-1 = Q\bullet u(z_s)u(z_s)^T,&\qquad \mbox{ for }s=1,\ldots,S \\
  & &&z_s \in \V
\end{aligned}
\end{equation*}

\begin{exmp}
  Let $\V\subset\C^9$ be the positive dimensional part of the cyclic $9$-roots problem.
  The cyclic $9$-roots equations are:
  {\footnotesize
  \begin{gather*}
x_1 + x_2 + x_3 + x_4 + x_5 + x_6 + x_7 + x_8 + x_9\\
x_1x_2 + x_2x_3 + x_3x_4 + x_4x_5 + x_5x_6 + x_6x_7 + x_7x_8 + x_8x_9 + x_9x_1\\
x_1x_2x_3 + x_2x_3x_4 + x_3x_4x_5 + x_4x_5x_6 + x_5x_6x_7 + x_6x_7x_8+ x_7x_8x_9  + x_8x_9x_1 + x_9x_1x_2\\
\vdots\\
x_1x_2x_3x_4x_5x_6x_7x_8 + x_2x_3x_4x_5x_6x_7x_8x_9 + \cdots +  x_9x_1x_2x_3x_4x_5x_6x_7\\
x_1x_2x_3x_4x_5x_6x_7x_8x_9 - 1
  \end{gather*}
}
  The zero set of these equations consists of a two-dimensional variety $\V$ of degree~$18$, and $6156$ isolated solutions~\cite{Faugere2001}.
  We remark that computing a Gr\"obner basis of these equations is very complicated unless its special structure is exploited.
  Indeed, Macaulay2 ran out of memory after $5$ hours of computation.

  We computed the irreducible decomposition of $\V$ using Bertini;
  it took $2h\, 45m$ with the default parameters. 
  The variety $\V$ decomposes into three pairs of conjugate irreducible varieties (each pair of degree~6).
  For each component we proceed to compute a sampling $2$-SOS certificate.
  We require $31$ complex samples on each component, which we obtained from Bertini in less than a second.
  Note that the upper bound from~\eqref{eq:nsamples} predicted $\frac{1}{2}\cdot \min\{{13\choose 4},\, 6 {6\choose 4}\}= 45$ samples.
  For each $j=0,\ldots,5$ we solved the respective sampling SDP, obtaining an infeasibility certificate of the form
  \begin{align*}
    -1 &= (R_j\,u(x))^T(R_j\,u(x)),\qquad \mbox{ for } x \in \V_j.
  \end{align*}
  This allows us to conclude that each irreducible component of $\V$ is purely complex.
  For instance, for the first irreducible component $\V_0$ it takes only $0.74s$ to obtain the certificate shown in Figure~\ref{fig:cycliccertificate}. 
\begin{figure}[p]
  \centering
  {\scriptsize 
    \begin{align*}
&\qquad\qquad u(x) = (x_8^2, x_7x_9, x_6^2, x_5x_9, x_5x_7, x_4^2, x_3x_6, x_2x_7, x_2x_6, x_2^2, x_1x_3, x_1^2, x_8, x_7, x_6, x_5, x_3, x_1, 1)\\
R_0&=\left[\begin{smallmatrix}\\
  -0.9638686  &-0.3445318  & 0.8395791  &-1.9531033  & 0.6329543  &-0.0152284  & 0.0238164  & 0.4701138  &-1.9766327  &-0.8363703\\
   0.3474835  &-0.3993919  & 0.5501348  &-1.2198730  &-0.2314149  & 0.0354563  & 1.0086575  & 0.4018444  & 1.0316339  &-0.6193326\\
   0.0117704  &-0.5278490  & 0.6157589  &-0.3131173  & 0.2207819  &-0.0080541  & 0.4038186  & 0.1500184  &-0.2618475  & 0.3089739\\
  -0.0131866  & 0.1597228  & 0.1191077  &-0.1088218  & 0.0697348  &-0.1149430  &-0.5067092  &-0.1883695  &-0.5993569  & 0.0244521\\
  -0.4504113  &-0.0761266  & 0.0056933  & 0.1535964  &-0.0860039  & 0.0007534  & 0.1264270  & 0.0880389  &-0.0927822  &-0.1429983\\
  -0.0804265  & 0.1450405  &-0.0077285  &-0.1657304  &-0.3240087  & 0.0014097  & 0.0631496  &-0.4083965  & 0.0191162  & 0.0854950\\
   0.0192110  & 0.1019831  &-0.1208989  &-0.1821975  &-0.1203214  & 0.0405222  & 0.0595267  &-0.1921851  &-0.0669972  &-0.1710978\\
  -0.1242984  & 0.1450764  &-0.2725352  &-0.1145423  & 0.0498037  & 0.0036466  &-0.0705293  & 0.2012444  & 0.0873671  &-0.2016367\\
  -0.0724047  &-0.0072012  &-0.0659910  & 0.1174698  &-0.0830511  & 0.0339559  & 0.1872153  &-0.0010648  &-0.1488432  & 0.1014557\\
  -0.1440253  & 0.0026597  &-0.1198142  & 0.0147434  & 0.1104305  & 0.0249783  & 0.0246079  &-0.0286915  &-0.0633959  & 0.0786306\\
   0.0872131  &-0.0503333  &-0.0426310  &-0.0108485  &-0.1538320  & 0.0351373  &-0.0895051  & 0.0994606  &-0.0858176  & 0.0033518\\
   0.0508126  & 0.0850471  &-0.0581353  &-0.0654513  & 0.0413446  &-0.0119532  & 0.0757765  & 0.0459333  &-0.0169990  & 0.1009677\\
   0.0112157  & 0.0251923  &-0.0096306  &-0.0680984  &-0.0005761  & 0.0111481  &-0.0373533  & 0.0293210  & 0.0134771  & 0.1119457\\
   0.0240035  &-0.1089833  &-0.0750114  &-0.0316807  & 0.0593823  &-0.0045386  &-0.0385441  &-0.0541272  & 0.0162211  & 0.0093965\\
   0.0727416  & 0.0067146  &-0.0258618  & 0.0206007  & 0.0284529  &-0.0125337  & 0.0450957  &-0.0142651  &-0.0460162  &-0.0377493\\
   0.0018262  & 0.0096039  & 0.0092749  &-0.0098153  & 0.0116513  & 0.0124708  &-0.0166840  &-0.0307406  & 0.0039079  &-0.0005090\\
   0.0167276  & 0.0418916  & 0.0339853  & 0.0127189  & 0.0353915  & 0.0352643  &-0.0230590  & 0.0037635  &-0.0020096  &-0.0118329\\
  -0.0000118  & 0.0028062  & 0.0010184  &-0.0000054  & 0.0000008  &-0.0076840  &-0.0004971  & 0.0000191  & 0.0000175  & 0.0000097\\
   0.0000403  & 0.0016845  &-0.0003547  &-0.0000443  &-0.0000077  & 0.0102089  & 0.0023837  &-0.0001760  &-0.0000343  &-0.0000479
\end{smallmatrix}\right.\\
&\qquad\qquad\left.\begin{smallmatrix}\\
   0.1130541  & 0.0243746  &-1.0601901  &-0.5184653  & 0.5389394  &-0.5290480  & 1.4666654  &-0.3021666  &-0.0722647\\
  -0.4838530  & 0.0446110  & 0.0443714  & 0.0400056  &-1.4678116  &-0.8155807  &-0.3859305  & 0.4715178  & 0.0326426\\
  -0.1549109  & 0.0903295  & 0.6966522  & 0.1005015  & 0.1763720  & 1.0574739  &-0.3501351  &-0.0462028  & 0.1140547\\
   0.5739993  &-0.0684158  & 0.3571626  & 0.0861604  &-0.6655387  &-0.1886137  &-0.4910517  & 0.1000489  & 0.1425230\\
  -0.0397509  & 0.0175350  &-0.4837186  & 0.1313369  & 0.0071916  & 0.1063667  &-0.5799628  & 0.0186656  &-0.1366172\\
   0.0288713  &-0.0270858  &-0.1711768  & 0.0516328  &-0.2256508  & 0.3277098  & 0.1917983  &-0.0453268  &-0.0388412\\
  -0.1128892  & 0.0088688  & 0.1873327  & 0.1850600  & 0.2445545  &-0.1280363  &-0.1585725  &-0.1034274  & 0.0225831\\
   0.0175299  &-0.0263491  & 0.0802817  &-0.1067258  &-0.0673210  & 0.2239093  & 0.0003892  &-0.0262654  & 0.1438365\\
  -0.1171573  & 0.0113433  & 0.0483356  &-0.1727556  &-0.1006414  &-0.0966047  &-0.0033401  &-0.1640771  & 0.2015432\\
  -0.1481463  & 0.0152680  & 0.0863339  & 0.1057026  &-0.0780024  &-0.0238758  & 0.0753908  & 0.1662874  &-0.0984320\\
  -0.0868955  & 0.0296646  &-0.0024721  &-0.0139792  &-0.0210549  & 0.0241345  &-0.0051360  & 0.0220213  &-0.0734852\\
   0.0428166  &-0.0263798  &-0.0151096  &-0.0341225  & 0.0015975  &-0.0045648  &-0.0459362  &-0.0584079  &-0.1230225\\
  -0.0239812  &-0.0019790  &-0.0630874  & 0.0477949  & 0.0312372  &-0.0118692  &-0.0260705  & 0.0443067  & 0.1156214\\
  -0.0084732  & 0.0434794  &-0.0322293  &-0.0031909  &-0.0252367  & 0.0031306  &-0.0173622  &-0.0663103  &-0.0011153\\
  -0.0013241  &-0.0085007  &-0.0350008  & 0.0115394  &-0.0048681  & 0.0243468  &-0.0006443  & 0.0436652  & 0.0255582\\
  -0.0261973  &-0.0090256  &-0.0019749  &-0.0604685  & 0.0112659  & 0.0022552  &-0.0232959  & 0.0276527  &-0.0063107\\
  -0.0469611  & 0.0047957  &-0.0132026  & 0.0243877  &-0.0253729  & 0.0101839  &-0.0006742  &-0.0451327  &-0.0022650\\
  -0.0020972  & 0.0071188  &-0.0000255  &-0.0008521  & 0.0000673  & 0.0000274  &-0.0002656  & 0.0000749  &-0.0000434\\
   0.0054816  & 0.0120243  & 0.0000808  &-0.0010943  & 0.0006406  & 0.0000603  & 0.0000083  & 0.0026580  & 0.0001099
\end{smallmatrix}\right]
  \end{align*}
}
  \caption{Positivstellensatz infeasibility certificate for the cyclic $9$-roots problem.}
  \label{fig:cycliccertificate}
\end{figure}

\end{exmp}

\subsection{Amoeba membership}
The (unlog) amoeba $\mathcal{A}_\V\subset \R_+^n$ of a variety $\V\subset \C^n$ is the image of $\V$ under the absolute value function, i.e.,
  $\mathcal{A}_\V = \{|z|: z\in \V \}$.
The amoeba membership problem is to determine whether some point $\lambda\in\R_+^n$ belongs to $\mathcal{A}_\V$.
Theobald and De Wolff recently proposed the use of Positivstellensatz certificates to prove that $\lambda\notin \mathcal{A}_\V$~\cite{Theobald2015}.
We now briefly describe this approach.

For some $f\in \C[z]$, let $\Re(f),\Im(f)\in \R[x,y]$ be such that 
\begin{align*}
f(x+i\,y)=\Re(f)(x,y)+i\,\Im(f)(x,y).
\end{align*}
Consider the following sets of equations in $\R[x,y]$:
\begin{align*}
  J_\V &:= \{ \Re(f_j),\Im(f_j)\}_{j=1}^{m}, &h_\lambda := \{ x_i^2+y_i^2-\lambda_i^2\}_{i=1}^n
\end{align*}
where $f_j$ are the defining equations of $\V$.
Theobald and De Wolff suggest computing a Gr\"obner basis of $J_\V\cup h_\lambda$ and then search for a Positivstellensatz infeasibility certificate.

Consider the following approach based on a set of samples $Z\subset \V$.
Let $\hat{\V}\in \C^{2n}$ be the zero set of $J_\V \subset \R[x,y]$.
Note that if $z\in \V$ then $(\Re(z),\Im(z))\in \hat{\V}$.
Thus, given some monomial vectors $u(x,y)$ and $v(x,y)$, we can formulate the following SDP:
\begin{equation*}
\begin{aligned}
  &\mbox{find} && Q\succeq 0,\;C \\
  &\mbox{subject to} &&-1 = Q\bullet u(x_s,y_s)u(x_s,y_s)^T + h_\lambda(x_s,y_s)^{T}C\,v(x_s,y_s),&\quad \mbox{ for }s=1,\ldots,S \\
  & &&z_s = x_s+i\,y_s \in \V
\end{aligned}
\end{equation*}

\begin{exmp}
  Let $\V\subset \C^{n k}$ be the complex variety associated to the Stiefel manifold $\mathit{St}(k,\R^n)$.
  Let $\lambda = (1/n,1/n,\ldots, 1/n)$, and let us show that $\lambda\notin \mathcal{A}_\V$ using the SDP from above. 
  We consider the degree~$1$ SOS relaxation for the case $n=6,k=4$.
  We require $1205$ complex samples on $\V$, which we obtain using the Cayley parametrization.
  It takes only $0.79s$ to compute the Positivstellensatz certificate from below. 
  On the other hand, Macaulay2 ran out of memory while computing a Gr\"obner basis of $J_\V$.
  {
  \begin{gather*}
    -1 = (R\,u(x,y))^T(R\,u(x,y)) - 1.2 \sum_{i=1}^6 h_i(x,y),\qquad \mbox{ for }(x,y)\in \hat{\V}\\
 u(x,y) = (y_6, y_5, y_4, y_3, y_2, y_1) \qquad h_i(x,y) = x_i^2+y_i^2 - 1/n^2\\
R=\left[\begin{smallmatrix}\\
   0.1765714  & 0.8458754  &-0.3371163  &-1.0598462  & 0.0269367  & 0.6447252\\
   0.2893688  & 0.1328983  &-1.4142041  & 0.4346374  & 0.1677938  &-0.2855976\\
  -0.4505154  &-0.6521358  &-0.3240160  & 0.2748310  &-0.0022626  & 1.2614402\\
   1.0819066  & 0.4199281  & 0.3317461  & 0.7231132  &-0.3725210  & 0.5304889\\
   0.8377745  &-1.0150421  &-0.1600336  &-0.6991182  &-0.3744590  &-0.1150085\\
   0.4579696  &-0.1868200  & 0.2138378  &-0.0250102  & 1.4464173  & 0.1299494
\end{smallmatrix}\right]
  \end{gather*}
}
\end{exmp}

\appendix
\section{Traditional SOS certificates}\label{s:sosvariety}
This section reviews two previously known methods to certify nonnegativity on a variety; 
we illustrate them with concrete, simple examples.
No new results are presented.

\subsection{Equations SOS}\label{s:classic}

Let $\V$ be a variety with defining equations $h=(h_1,\ldots,h_m)$, and let $p\in \R[x]$ be nonnegative on $\V\cap\R^n$.
The standard approach to certify this nonnegativity is to compute an SOS polynomial $F$ and a tuple of polynomials $g=(g_1,\ldots,g_m)$ that satisfy equation~\eqref{eq:sdprelaxation}.
Let $u(x)\in\R[x]^N$ consist of all monomials up to degree~$d$.
Computing an equations $d$-SOS certificate $(F,g)$ reduces to the following problem:
\begin{equation}\label{eq:SDPclassic}
\begin{aligned}
  &\mbox{find } && Q\in \SR^{N},\quad C\in\R^{m\times N},\quad Q\succeq 0 \\
  &\mbox{subject to} &&p(x) = Q\bullet u(x)u(x)^T + h(x)^TC\,u(x)
\end{aligned}
\end{equation}
where $F(x) = Q\bullet u(x)u(x)^T$ and $g(x) = C\,u(x)$.

\begin{exmp}\label{exmp:SO2}
  Let us retake the case from Example~\ref{exmp:SO2_samples}.
  We want to certify that $p(X) = 4X_{21}-2X_{11}X_{22}-2X_{12}X_{21}+3$ is nonnegative on the variety in~\eqref{eq:SO2}.
  We fix a degree bound of $d=1$ and let $u(X)=(1,X_{11},X_{12},X_{21},X_{22})$.
  The dimensions of the matrices in the SDP~\eqref{eq:SDPclassic} are $Q\in \SR^{5}$, $C\in \R^{4\times 5}$.
  Note that these dimensions are larger than in Example~\ref{exmp:SO2_samples}.
  Solving the SDP leads to the following SOS certificate:
  \begin{gather*}
    p(X) = F(X) - g_1h_1 - g_2h_2 + g_3h_3 - g_4h_4,\\
    F(X) = (X_{21} - X_{12} + 1)^2 + (X_{11} - X_{22})^2, \\
    g_1 = X_{12}+1, \quad
    g_2 = X_{21}+1, \quad
    g_3 = X_{11}+X_{22}, \quad
    g_4 = X_{21}+X_{12} \\
    h_1 = X_{11}^2+X_{21}^2-1, \quad
    h_2 = X_{12}^2+X_{22}^2-1, \quad
    h_3 = X_{11}X_{12}+X_{21}X_{22}, \quad
    h_4 = \det(X)-1.
  \end{gather*}
\end{exmp}

\subsection{Quotient ring SOS}\label{s:quotientring}

It is possible to take advantage of the quotient ring structure to obtain a simpler SDP.
The standard approach to do so requires a Gr\"obner basis of the ideal $I:=\langle h\rangle$.
We briefly explain the procedure now.
We refer to~\cite{clo} for an introduction to Gr\"obner bases.

Consider the quotient ring $\mathcal{R}'= \R[x]/I$, and let $\phi: \R[x]\to \mathcal{R}'$ be the canonical homomorphism.
Since $\phi(g_jh_j)=0$, when we view the SDP in~\eqref{eq:SDPclassic} in the quotient ring we obtain:
\begin{equation}\label{eq:SDPgrobner}
\begin{aligned}
  &\mbox{find } &&Q\in \R^{SN},\quad Q\succeq 0 \\
  &\mbox{subject to} &&\phi(p(x)) = Q\bullet\phi(u(x)u(x)^T)
\end{aligned}
\end{equation}
Note that the matrix $C$ was eliminated.
The above SDP requires methods to represent and compute the quotient ring $\mathcal{R}'$.
Given a Gr\"obner basis $\mathit{gb}$ of $I$, there is a simple way to achieve this.
Concretely, any polynomial $f\in \R[x]$ can be written in \emph{normal form}, denoted as $\phi_{\mathit{gb}}(f)$, with respect to $\mathit{gb}$.
This normal form map $\phi_{\mathit{gb}}$ is effectively representing the quotient ring.
In addition, there is a natural monomial basis $u(x)$ to use, given by the standard monomials with respect to $\mathit{gb}$.

\begin{exmp}\label{exmp:SO2_Grobner}
  Consider again the case from Example~\ref{exmp:SO2_samples}.
  Consider the Gr\"obner basis
  \begin{align*}
    \mathit{gb} = \{\underline{X_{11}}-X_{22},\;
    \underline{X_{12}}+X_{21}, \;
    \underline{X_{21}^2}+X_{22}^2-1\},
  \end{align*}
  where the leading monomials are underlined.
  Let $u(X)=(1,X_{21},X_{22})$ be the standard monomials of degree at most~$d=1$. 
  Computing the normal form $\phi_{\mathit{gb}}(u(X)u(X)^T)$, we obtain the SDP:
\begin{eqnarray*}
  \mbox{find }\quad & Q\in \SR^{3},\quad Q\succeq 0 \\
  \mbox{subject to }\quad & 4X_{21}-4X_{22}^2+5 = Q \bullet 
    \begin{bmatrix} 1 &X_{21} &X_{22} \\ X_{21} &1-X_{22}^2  &X_{21}X_{22} \\ X_{22} &X_{21}X_{22}  &X_{22}^2 \end{bmatrix}\\
\end{eqnarray*}
  Solving the SDP leads to the SOS certificate $\phi_{\mathit{gb}}(p(X))= \phi_{\mathit{gb}}((2X_{21}+1)^2)$.
  Note that the certificate obtained, as well as the dimension of $Q$, agrees with that of Example~\ref{exmp:SO2_samples}.
\end{exmp}

\begin{rem}[Sampling SOS]
  Our sampling approach can be seen as a quotient ring formulation.
  The difference is that we use the radical ideal $J = \sqrt{I}$ and the coordinate ring $\mathcal{R}=\R[x]/J$.
  Given a sample set $Z=\{z_1,\ldots,z_S\}$ our description of $\mathcal{R}$ is given by the \emph{evaluation map} 
    $\phi_Z(f) := (f(z_1),\ldots,f(z_S))$.
\end{rem}

\bibliography{refernces}

\begin{thebibliography}{10}

\bibitem{bertini}
Daniel~J Bates, Jonathan~D Hauenstein, Andrew~J Sommese, and Charles~W Wampler.
\newblock Bertini: Software for numerical algebraic geometry.
\newblock Available at \url{www.nd.edu/~sommese/bertini}, 2006.

\bibitem{Bates2013}
Daniel~J Bates, Jonathan~D Hauenstein, Andrew~J Sommese, and Charles~W Wampler.
\newblock {\em Numerically solving polynomial systems with Bertini}, volume~25
  of {\em Software, Environments, and Tools}.
\newblock SIAM, 2013.

\bibitem{dsdp}
Steven~J Benson and Yinyu Ye.
\newblock {DSDP5}: Software for semidefinite programming.
\newblock {\em Mathematics and Computer Science Division, Argonne National
  Laboratory, Argonne, IL, Tech. Rep. ANL/MCS-P1289-0905}, 2005.

\bibitem{Blekherman2013}
Grigoriy Blekherman, Pablo~A. Parrilo, and Rekha~R. Thomas.
\newblock {\em Semidefinite optimization and convex algebraic geometry},
  volume~13 of {\em Series on Optimization}.
\newblock MOS-SIAM, 2013.

\bibitem{Blekherman2016}
Grigoriy Blekherman, Gregory Smith, and Mauricio Velasco.
\newblock Sums of squares and varieties of minimal degree.
\newblock {\em Journal of the American Mathematical Society}, 29(3):893--913,
  2016.

\bibitem{Chardin1989}
Marc Chardin.
\newblock Une majoration de la fonction de {H}ilbert et ses cons{\'e}quences
  pour l'interpolation alg{\'e}brique.
\newblock {\em Bulletin de la Soci{\'e}t{\'e} Math{\'e}matique de France},
  117(3):305--318, 1989.

\bibitem{Chu2001}
Moody~T Chu and Nickolay~T Trendafilov.
\newblock The orthogonally constrained regression revisited.
\newblock {\em Journal of Computational and Graphical Statistics},
  10(4):746--771, 2001.

\bibitem{Ciliberto:2001aa}
Ciro Ciliberto.
\newblock Geometric aspects of polynomial interpolation in more variables and
  of {W}aring's problem.
\newblock In {\em European Congress of Mathematics}, volume 201 of {\em
  Progress in Mathematics}, pages 289--316. Springer, 2001.

\bibitem{clo}
David~A Cox, John Little, and Donal O'Shea.
\newblock {\em Ideals, varieties, and algorithms: an introduction to
  computational algebraic geometry and commutative algebra}.
\newblock Springer, 2007.

\bibitem{De-Lathauwer:2000aa}
Lieven De~Lathauwer, Bart De~Moor, and Joos Vandewalle.
\newblock On the best rank-1 and rank-$(r_1, r_2, \ldots, r_n)$ approximation
  of higher-order tensors.
\newblock {\em SIAM Journal on Matrix Analysis and Applications},
  21(4):1324--1342, 2000.

\bibitem{Faugere2001}
Jean-Charles Faug{\`e}re.
\newblock Finding all the solutions of {C}yclic 9 using {G}r{\"o}bner basis
  techniques.
\newblock In {\em Computer Mathematics - Proceedings of the Fifth Asian
  Symposium (ASCM 2001)}, volume~9, pages 1--12. World Scientific, 2001.

\bibitem{macaulay2}
Daniel~R. Grayson and Michael~E. Stillman.
\newblock Macaulay2, a software system for research in algebraic geometry.
\newblock Available at \url{http://www.math.uiuc.edu/Macaulay2/}.

\bibitem{Jantzen2004}
Jens~Carsten Jantzen.
\newblock Nilpotent orbits in representation theory.
\newblock In {\em Lie theory}, pages 1--211. Springer, 2004.

\bibitem{Kabanets2004}
Valentine Kabanets and Russell Impagliazzo.
\newblock Derandomizing polynomial identity tests means proving circuit lower
  bounds.
\newblock {\em Computational Complexity}, 13(1-2):1--46, 2004.

\bibitem{Landsberg:2007aa}
Joseph~M Landsberg and Jerzy Weyman.
\newblock On the ideals and singularities of secant varieties of {S}egre
  varieties.
\newblock {\em Bulletin of the London Mathematical Society}, 2007.

\bibitem{Lasserre2001}
Jean~B Lasserre.
\newblock Global optimization with polynomials and the problem of moments.
\newblock {\em SIAM Journal on Optimization}, 11(3):796--817, 2001.

\bibitem{Lasserre2009}
Jean~B Lasserre.
\newblock {\em Moments, positive polynomials and their applications}, volume~1
  of {\em Series on Optimization and Its Applications}.
\newblock World Scientific, 2009.

\bibitem{laurent2009sums}
Monique Laurent.
\newblock Sums of squares, moment matrices and optimization over polynomials.
\newblock In {\em Emerging applications of algebraic geometry}, pages 157--270.
  Springer, 2009.

\bibitem{Liu2007}
Zhang Liu and Lieven Vandenberghe.
\newblock Low-rank structure in semidefinite programs derived from the {KYP}
  lemma.
\newblock In {\em 46th IEEE Conference on Decision and Control}, pages
  5652--5659, 2007.

\bibitem{Lofberg2004}
Johan L{\"o}fberg and Pablo~A Parrilo.
\newblock From coefficients to samples: a new approach to {SOS} optimization.
\newblock In {\em 43rd IEEE Conference on Decision and Control}, volume~3,
  pages 3154--3159, 2004.

\bibitem{Nie2006}
Jiawang Nie, James Demmel, and Bernd Sturmfels.
\newblock Minimizing polynomials via sum of squares over the gradient ideal.
\newblock {\em Mathematical programming}, 106(3):587--606, 2006.

\bibitem{Parrilo2003b}
Pablo~A Parrilo.
\newblock Exploiting structure in sum of squares programs.
\newblock In {\em 42nd IEEE Conference on Decision and Control}, volume~5,
  pages 4664--4669, 2003.

\bibitem{Parrilo2003}
Pablo~A Parrilo.
\newblock Semidefinite programming relaxations for semialgebraic problems.
\newblock {\em Mathematical programming}, 96(2):293--320, 2003.

\bibitem{Permenter:2012aa}
Frank Permenter and Pablo~A Parrilo.
\newblock Selecting a monomial basis for sums of squares programming over a
  quotient ring.
\newblock In {\em IEEE 51st Annual Conference on Decision and Control}, pages
  1871--1876, 2012.

\bibitem{Roh2007}
Tae Roh, Bogdan Dumitrescu, and Lieven Vandenberghe.
\newblock Multidimensional {FIR} filter design via trigonometric sum-of-squares
  optimization.
\newblock {\em IEEE Journal of Selected Topics in Signal Processing},
  1(4):641--650, 2007.

\bibitem{Roh2006}
Tae Roh and Lieven Vandenberghe.
\newblock Discrete transforms, semidefinite programming, and sum-of-squares
  representations of nonnegative polynomials.
\newblock {\em SIAM Journal on Optimization}, 16(4):939--964, 2006.

\bibitem{Sauer2006}
Tomas Sauer.
\newblock Polynomial interpolation in several variables: lattices, differences,
  and ideals.
\newblock {\em Studies in Computational Mathematics}, 12:191--230, 2006.

\bibitem{Saxena2009}
Nitin Saxena.
\newblock Progress on polynomial identity testing.
\newblock {\em Bulletin of the EATCS}, 99:49--79, 2009.

\bibitem{Scheiderer2009}
Claus Scheiderer.
\newblock Positivity and sums of squares: a guide to recent results.
\newblock In {\em Emerging applications of algebraic geometry}, pages 271--324.
  Springer, 2009.

\bibitem{Schmudgen1991}
Konrad Schm{\"u}dgen.
\newblock The $k$-moment problem for compact semi-algebraic sets.
\newblock {\em Mathematische Annalen}, 289(1):203--206, 1991.

\bibitem{Schweighofer2006}
Markus Schweighofer.
\newblock Global optimization of polynomials using gradient tentacles and sums
  of squares.
\newblock {\em SIAM Journal on Optimization}, 17(3):920--942, 2006.

\bibitem{Sommese2005}
Andrew~John Sommese and Charles~Weldon Wampler.
\newblock {\em The Numerical solution of systems of polynomials arising in
  engineering and science}, volume~99.
\newblock World Scientific, 2005.

\bibitem{Theobald2015}
Thorsten Theobald and Timo De~Wolff.
\newblock Approximating amoebas and coamoebas by sums of squares.
\newblock {\em Mathematics of Computation}, 84(291):455--473, 2015.

\bibitem{sdpt3}
Reha~H T\"ut\"unc\"u, Kim~C Toh, and Michael~J Todd.
\newblock Solving semidefinite-quadratic-linear programs using {SDPT3}.
\newblock {\em Mathematical programming}, 95(2):189--217, 2003.

\bibitem{phcpack}
Jan Verschelde.
\newblock Algorithm 795: {PHC}pack: {A} general-purpose solver for polynomial
  systems by homotopy continuation.
\newblock {\em ACM Transactions on Mathematical Software (TOMS)},
  25(2):251--276, 1999.

\bibitem{Viklands:2006aa}
Thomas Viklands.
\newblock {\em Algorithms for the weighted orthogonal {P}rocrustes problem and
  other least squares problems}.
\newblock PhD thesis, Umea University, Sweden, 2006.

\bibitem{Wang:2007aa}
Huan Wang, Shuicheng Yan, Dong Xu, Xiaoou Tang, and Thomas Huang.
\newblock Trace ratio vs.\ ratio trace for dimensionality reduction.
\newblock In {\em IEEE Conference on Computer Vision and Pattern Recognition},
  pages 1--8, 2007.

\bibitem{Zhang2014}
Lei~Hong Zhang and Ren~Cang Li.
\newblock Maximization of the sum of the trace ratio on the {S}tiefel manifold,
  {I}: {T}heory.
\newblock {\em Science China Mathematics}, 57:2495--2508, 2014.

\bibitem{Zhang2014a}
Lei~Hong Zhang and Ren~Cang Li.
\newblock Maximization of the sum of the trace ratio on the {S}tiefel manifold,
  {II}: {C}omputation.
\newblock {\em Science China Mathematics}, 57:1--18, 2014.

\end{thebibliography}
\bibliographystyle{plain}

\end{document}